\documentclass[11pt]{article}

\usepackage{fullpage}
\usepackage{authblk}
\usepackage{natbib}
\usepackage{algorithm}
\usepackage{algpseudocode}
\usepackage{amsmath,amsthm,amsfonts}
\usepackage{amsmath}
\usepackage[subnum]{cases}
\usepackage{tcolorbox}
\usepackage{mathrsfs}
\usepackage{amssymb}
\usepackage{color}
\usepackage{mathrsfs}
\usepackage{empheq}
\usepackage{enumitem}
\usepackage{bm}
\usepackage{multirow}
\usepackage{booktabs}
\usepackage{makecell}
\usepackage{graphicx}
\usepackage{subcaption}
\usepackage{comment}
\usepackage{cases}
\usepackage{appendix}
\usepackage{tikz}
\usetikzlibrary{arrows,shapes}
\usepackage[colorlinks= true, linkcolor=red, citecolor=blue, urlcolor=black]{hyperref}
\usepackage{cleveref}
\usepackage{marginnote}
\usepackage{diagbox} 

\numberwithin{equation}{section}

\newtheorem{theorem}{Theorem}[section]
\newtheorem{lemma}[theorem]{Lemma}

\newtheorem{prop}[theorem]{Proposition}
\newtheorem{defn}[theorem]{Definition}
\newtheorem{assumption}[theorem]{Assumption}

\newcommand{\norm}[1]{\left\|#1\right\|}
\newcommand{\inner}[1]{\left\langle #1\right\rangle}
\newcommand{\dd}{\mathrm{d}}

\DeclareMathOperator*{\argmin}{arg\,min}

\allowdisplaybreaks

\makeatletter
\newcommand{\leqnomode}{\tagsleft@true}
\newcommand{\reqnomode}{\tagsleft@false}
\makeatother

\title{Convergent Proximal Multiblock ADMM for Nonconvex Dynamics-Constrained Optimization}
\author[1,2]{Bowen Li \thanks{Corresponding author: \url{bowen@lsec.cc.ac.cn} } }
\author[1]{Ya-xiang Yuan}
\affil[1]{Academy of Mathematics and Systems Science, Chinese Academy of Sciences, Beijing 100190, China}
\affil[2]{School of Mathematical Sciences, University of Chinese Academy of Sciences, Beijing 100049, China}
\date\today

\begin{document}

\maketitle

\begin{abstract}
This paper proposes a provably convergent multiblock ADMM for nonconvex optimization with nonlinear dynamics constraints, overcoming the divergence issue in classical extensions. We consider a class of optimization problems that arise from discretization of dynamics-constrained variational problems that are optimization problems for a functional constrained by time-dependent ODEs or PDEs. This is a family of $n$-sum nonconvex optimization problems with nonlinear constraints. We study the convergence properties of the proximal alternating direction method of multipliers (proximal ADMM) applied to those problems. Taking the advantage of the special problem structure, we show that under local Lipschitz and local $L$-smooth conditions, the sequence generated by the proximal ADMM is bounded and all accumulation points are KKT points. Based on our analysis, we also design a procedure to determine the penalty parameters $\rho_i$ and the proximal parameters $\eta_i$. We further prove that among all the subsequences that converge, the fast one converges at the rate of $o(1/k)$. The numerical experiments are performed on 4D variational data assimilation problems and as the solver of implicit schemes for stiff problems. The proposed proximal ADMM has more stable performance than gradient-based methods. We discuss the implementation to solve the subproblems, a new way to solve the implicit schemes, and the advantages of the proposed algorithm.

\end{abstract}

%

\section{Introduction}
Nonconvex optimization problems with nonlinear equality constraints arise in many areas of application. From scientific problems to engineering techniques, these problems are formulated as time-dependent ODE-constrained optimization problems or PDE-constrained optimization problems, such as the strongly constraint 4D variational data assimilation (4DVar) for numerical weather prediction~\cite{le1986variational} and flow control in~\cite{gunzburger2002perspectives}. An important problem is the state estimation for a dynamical system of finite or infinite dimension, which can be stated as the following variational problem 
\begin{subequations}\label{eq:con_optpro}
\begin{align}
    \min_{X(t) \in \mathcal{H}}\ & \int_0^T F(X(t),t) \dd t + G(X(0)),\\
    \mathrm{s.\ t.}\ & \dot{X}(t) = \Phi(X(t)) \label{eq:con_constraints},
\end{align}
\end{subequations}
where $\mathcal{H}$ is some function space. The functions (or functionals for infinite-dimension dynamics) $F, G$ are given and are supposed to be smooth. The operator $\Phi$ maps a Hilbert space (of finite or infinite dimension) to itself. Usually $F$ is called the cost function, and $G$ serves as the regularization term. Although the dynamics~\eqref{eq:con_constraints} under consideration is autonomous, the results of this article can easily be generalized to non-autonomous systems.
We consider a family of nonconvex optimization problems with nonlinear constraints, which originates from the discretization of problem~\eqref{eq:con_optpro} and explicit scheme of dynamics~\eqref{eq:con_constraints}, formulated as following:
\begin{subequations}\label{eq:optpro}
\begin{align}
    \min_x\ & f(x) := \sum_{i=0}^n f_i(x_i),\\
    \mathrm{s.\ t.}\ & x_{j+1} = \varphi(x_j),\ j=0, \ldots, n-1,
\end{align}
\end{subequations}
where $x_i \in \mathbb{R}^d$ and $x = (x_0, \ldots, x_n)^T\in \mathbb{R}^{(n+1)d}$. $f_i: \mathbb{R}^d \to \mathbb{R}$ for $i=0, \ldots, n$ are objective functions and $\varphi: \mathbb{R}^d \to \mathbb{R}^d$ is the map of discrete dynamics. We assume that all maps $f_i$ and $\varphi$ are continuously differentiable. The problem~\eqref{eq:optpro} is related to the trajectory optimization that has been studied in~\cite{kelly2017introduction, betts1998survey}.

The problem~\eqref{eq:optpro} is a constrained optimization problem. Any constrained optimization methods such as sequential quadratic programming (SQP) can be applied to it, as discussed in the literature~\cite{tenny2004nonlinear, betts2010practical, bryson1975applied}. However, problem~\eqref{eq:optpro}has the special $n$-sum structure of its objective function and dynamics constraints. We should design specific algorithms that utilize this structure for better performance. One of the classic methods to solve this problem is the adjoint method~\cite{talagrand1987variational, courtier1987variational}. This method provides a fast way to compute the gradient of problem~\eqref{eq:optpro} with respect to the initial value $x_0$. Then the conjugate gradient method or the quasi-Newton method can be applied to find a solution of~\eqref{eq:optpro}. However, for nonconvex optimization problems, the critical point found by gradient methods usually depends sensitively on the initial guess that is given before running the algorithm. We consider here a different approach. The proximal alternating direction method of multipliers (ADMM) is applied to solve the problem~\eqref{eq:optpro}. The ADMM method is more flexible because it has some parameters to adjust. Usually, different parameter settings will lead the algorithm to converge to different critical points. With some guidelines, there is more possibility to find the solution we want. In addition, the proximal ADMM is not sensitive to the initial guess once we fixed the parameters. We provide a detailed comparison in Section~\ref{sec:num_exp}.

The proximal ADMM depends on the augmented Lagrangian of~\eqref{eq:optpro}, which is defined as 
\begin{subequations} \label{eq:aug_lag}
\begin{align}
    L_\rho(x, \lambda) = & \sum_{i=0}^n f_i(x_i) + \sum_{i=0}^{n-1}\left( \inner{\lambda_i, x_{i+1} - \varphi(x_i)} + \frac{\rho_i}{2}\norm{x_{i+1} - \varphi(x_i)}^2 \right) \\
    ={} & \sum_{i=0}^n f_i(x_i) + \sum_{i=0}^{n-1}\left( \frac{\rho_i}{2}\norm{x_{i+1} - \varphi(x_i) + \frac{1}{\rho_i}\lambda_i}^2 - \frac{1}{2\rho_i}\norm{\lambda_i}^2 \right), \label{eq:aug_lag2}
\end{align}
\end{subequations}
where $\lambda = (\lambda_0, \ldots, \lambda_{n-1})^T \in \mathbb{R}^{(n+1)d}$ are the Lagrange multipliers, which is also called the dual variables, and $\rho = (\rho_0, \ldots,  \rho_{n-1})$ are the penalty parameters. When $\rho_i \equiv 0$ for all $i=0, \ldots, n-1$. The function $L_\rho(x, \lambda)$ becomes 
\begin{equation}\label{eq:lag}
    L(x, \lambda) := \sum_{i=0}^n f_i(x_i) + \sum_{i=0}^{n-1} \inner{\lambda_i, x_{i+1} - \varphi(x_i)},
\end{equation}
which is the Lagrangian of problem~\eqref{eq:optpro}. The proximal ADMM method is 
\begin{equation}\label{eq:admm_ori}
\begin{cases}
    x_i^{k+1} = \argmin \left\{ L_\rho(x_{<i}^{k+1}, x_i, x_{>i}^k, \lambda^k) + \frac{1}{2\eta_i}\norm{x_i - x_i^k}^2 \right\}, & \text{ for } i=0, \ldots, n,\\
    \lambda_j^{k+1} = \lambda_j^k + \rho_j\left( x_{j+1}^{k+1} - \varphi(x_j^{k+1}) \right), & \text{ for } j=0, \ldots, n-1,
\end{cases}
\end{equation}
where $\eta = (\eta_0, \ldots, \eta_n)$ are the parameters that control the weights of the proximal terms. These parameters can also be viewed as the step sizes for each subproblem. We denote $x_{<i} = (x_0, \ldots, x_{i-1})$ and $x_{>i} = (x_{i+1}, \ldots, x_n)$ for simplicity. Substituting~\eqref{eq:aug_lag2} into~\eqref{eq:admm_ori} and omitting all constant term in each subproblem, we obtain 
\begin{equation} \label{eq:admm}
    \begin{cases}
        x_0^{k+1} = \argmin\limits_{x_0}\left\{f_0(x_0) + \frac{\rho_0}{2}\norm{x_1^k - \varphi(x_0) + \frac{1}{\rho_0}\lambda_0^k}^2 + \frac{1}{2\eta_0}\norm{x_0 - x_0^k}^2\right\}, \\
        x_i^{k+1} = \argmin\limits_{x_i}\left\{ f_i(x_i) + \frac{\rho_{i-1}}{2}\norm{x_i - \varphi(x_{i-1}^{k+1}) + \frac{1}{\rho_{i-1}}\lambda_{i-1}^k}^2 \right.\\
        \phantom{x_0^{k+1} = \argmin\{\}} \left. + \frac{\rho_i}{2} \norm{x_{i+1}^k - \varphi(x_i) + \frac{1}{\rho_i}\lambda_i^k}^2 + \frac{1}{2\eta_i}\norm{x_i - x_i^k}^2 \right\}, \text{ for } i = 1, \ldots, n-1,\\
        x_n^{k+1} = \argmin\limits_{x_n} \left\{ f_n(x_n) + \frac{\rho_{n-1}}{2}\norm{x_n - \varphi(x_{n-1}) + \frac{1}{\rho_{n-1}}\lambda_{n-1}^k}^2 + \frac{1}{2\eta_n}\norm{x_n - x_n^k}^2 \right\}, \\
        \lambda_j^{k+1} = \lambda_j^k + \rho_j\left(x_{j+1}^{k+1} - \varphi(x_j^{k+1})\right), \text{ for } j=0, \ldots, n-1.
    \end{cases}
\end{equation}

The alternating direction method of multipliers (ADMM) algorithm was first proposed in~\cite{glowinski1975approximation} as an operator splitting method. The original ADMM works for the case where $n=2$ and the constraints are linear functions. However, the direct extension from $n=2$ to $n>2$ does not necessarily converge when applying the original ADMM without the proximal term, even for convex optimization~\cite{chen2016direct}. The ADMM algorithm is closely related to the augmented Lagrangian method (ALM) which originates from the work of Hestenes~\cite{hestenes1969multiplier} and Powell~\cite{powell1969method}. The properties of this algorithm can be found in the classic textbook by Bertsekas~\cite{bertsekas1982constrained}. Rockafellar~\cite{rockafellar1976augmented} proposed the proximal ALM algorithm by adding a proximal term to the subproblem. This method has advantages in both theory and applications, especially for nonconvex optimization. The ADMM can be viewed as the alternating coordinate descent version of primal descent in the ALM. For this reason, the proximal term can be introduced to each subproblem of the ADMM algorithm, which can be viewed as an alternating coordinate descent version of the proximal ALM algorithm. This motivates the algorithm~\eqref{eq:admm_ori}.

\subsection{Related work}
\textbf{Convex optimization with linear constraints.} Over the past two decades, numerous research works have been devoted to understanding the behavior of the ADMM algorithm for convex optimization, which are problems with convex objective functions and linear constraints. It has become very popular for the large-scale problems that arise in statistical learning, signal processing, computer vision, etc.~\cite{boyd2011distributed, yin2008bregman, schizas2007consensus, zhang2011unified}. This promising computational performance stimulates a large interest in studying the ADMM algorithm theoretically. For the case where $n=2$, the algorithm converges with very mild conditions, see~\cite{bertsekas1989parallel, boyd2011distributed}. Under the same conditions, in~\cite{he20121, monteiro2013iteration, davis2017convergence} the authors showed that the ADMM converges at a rate of $o(1/k)$. Although the direct extension of ADMM to multiblock problems does not converge as suggested in~\cite{chen2016direct}, there is much research on modifying the original algorithm such that it has convergence guarantees for optimization problems with $n \ge 3$. The extension has two types, the Gauss-Seidel decomposition and the Jacobian decomposition. The first uses the updated primal variables immediately in the current iteration as the algorithm~\eqref{eq:admm_ori}, while the second uses only the information of the previous iteration. In~\cite{he2012alternating, he2015full} the authors applied a prediction-correction approach to obtain convergent modified ADMM algorithms for multiblock problems, and in~\cite{deng2017parallel} the authors added the proximal term to each subproblem and proved the convergence of the proposed algorithm.

\textbf{Nonconvex optimization with linear constraints.}
Recently, there is increasing interest in proving the convergence results of ADMM for nonconvex optimization. The first step is to consider problems with nonconvex objectives with linear constraints. Many applications demonstrate that the ADMM algorithm also performs well in this case, such as phase retrieval~\cite{wen2012alternating}, distributed matrix factorization~\cite{zhang2014asynchronous}, tensor decomposition~\cite{liavas2015parallel}, etc. In~\cite{themelis2020douglas} the authors proved the convergence of the original ADMM algorithms for case $n=2$. The papers~\cite{hong2016convergence} and~\cite{wang2019global} show the convergence of the direct extension of ADMM with Gauss-Seidel decomposition for some certain family of nonconvex optimization problems. In~\cite{yang2022proximal} the authors proved the convergence of the proximal ADMM algorithm with Jacobian decomposition for nonconvex optimization problems with linear constraints.

\textbf{Nonconvex optimization with nonlinear constraints.} Due to the technical difficulty, there is even fewer research on the convergence analysis of the ADMM algorithm for nonconvex optimization problems with nonlinear constraints. In~\cite{xie2021complexity} the authors proposed a method to show the convergence of the proximal ALM algorithm for nonconvex optimization with nonlinear constraints. This paper is the main source of inspiration for our work. In~\cite{cohen2022dynamic}, the authors proved the convergence for case $n=2$ of a linearized proximal ADMM algorithm with a stronger assumption. Our work can be viewed as a generalization of this result. There are also attempts at modifying the proximal ADMM algorithm to obtain the convergence property. In recent work~\cite{sun2024dual}, the authors proposed the dual descent rule to update the dual variable in each iteration. The ADMM algorithm attempts to find the saddle point of the augmented Lagrangian. It is related to some recent work on designing the appropriate algorithm for nonconvex-concave minimax problems, such as~\cite{xu2024derivative}. The work~\cite{li2025numerical} shows the performance of the linearized proximal ADMM algorithm applied to data variational problems.

\subsection{Contributions}
We consider a class of nonconvex optimization problems with nonlinear constraints that come from the discretization of dynamics-constrained variational problems. Unlike traditional adjoint method, we propose a new framework that applies the proximal ADMM algorithm to solve this family of problems.

Exploiting the special structure of the problems~\eqref{eq:optpro}, we introduce an induction technique, which enables us to generalize the analysis in~\cite{xie2021complexity} and~\cite{cohen2022dynamic} to the multiblock proximal ADMM algorithm. This also overcomes the difficulty of extending the result in~\cite{hong2016convergence, wang2019global} to nonlinear dynamics constraints. For differentiable functions that are local Lipschitz and local $L$-smooth with only mild additional assumptions, we show that the sequence $(x^k, \lambda^k)$ generated by the algorithm~\eqref{eq:admm} stays in a bounded domain depending on the initial point and parameters, and any accumulation points of the sequence are KKT points of problem~\eqref{eq:optpro}. Based on the proof, we design a procedure for successively choosing the penalty parameters $\rho_i$ and the proximal parameters $\eta_i$ that guarantee convergence. Furthermore, we show that among all subsequences that converge, the fast convergence rate is $o(1/k)$, where $k$ is the iteration number.

In addition, we discuss several generalizations of the problem~\eqref{eq:optpro}, to which the convergence analysis framework we proposed applied without a major difference. These include non-autonomous dynamics, semi-implicit schemes, implicit schemes, and optimal control problems.

Finally, we perform numerical experiments on 4D data assimilation problems and on solving implicit schemes of stiff equations. The implementation of the solvers for the subproblems and a new way to solve the implicit schemes as an optimization problem rather than nonlinear equations are discussed in detail.

\subsection{Organization}
The remainder of this paper is organized as follows. Section~\ref{sec:pre} presents some preliminary assumptions, definitions, and propositions used in this paper. In Section~\ref{sec:con_ana}, the main convergence analysis with detailed proofs is presented and the generalization of our convergence results to a wider range of optimization problems are discussed. In Section~\ref{sec:num_exp}, we apply the proximal ADMM algorithm to some typical applications and report the numerical results. We summarize the paper and discuss future work in Section~\ref{sec:conclusion}.

\section{Preliminaries} \label{sec:pre}

In this section, we discuss the assumptions and some background that are necessary in the convergence analysis of the proximal ADMM algorithm~\eqref{eq:admm}.

Let $\norm{\cdot}$ denote the Euclidean norm of a vector and the induced $l_2$ norm of a matrix. Throughout the paper, we assume the following.
\begin{assumption}\label{asm:level_set}
    There exist parameters $\rho_i^0 > 0$ for $i=0, \ldots, n-1$, such that the function $\sum_{i=0}^n f_i(x_i) + \sum_{i=0}^{n-1}\frac{\rho^0_i}{2}\norm{x_{i+1} - \varphi(x_i)}^2$ has bounded level set, that is, for any $\alpha \in \mathbb{R}$ the set 
    \begin{equation}\label{eq:level_set}
        S_\alpha := \left\{x=(x_0,\ldots, x_n) \Bigg| \sum_{i=0}^n f_i(x_i) + \sum_{i=0}^{n-1}\frac{\rho^0_i}{2}\norm{x_{i+1} - \varphi(x_i)}^2 \le \alpha \right\}
    \end{equation}
    is empty or bounded.
\end{assumption}
This assumption is similar to those of~\cite{xie2021complexity}, which is a weaker assumption than the boundedness of level sets of the objective function. As the functions $f_i$ and $\varphi$ are continuous, an immediate corollary is that the function $\sum_{i=0}^n f_i(x_i) + \sum_{i=0}^{n-1}\frac{\rho^0_i}{2}\norm{x_{i+1} - \varphi(x_i)}^2$ is lower-bounded by some constant $B_\alpha$ on the set $S_\alpha$. We have the following sufficient condition for the above assumption to hold.
\begin{prop}
    Suppose that the functions $f_i$ are bounded below for $i=0, \ldots, n$. In addition, assume that $f_0(x_0) \to \infty$ if $\norm{x_0} \to \infty$ and that the map $\varphi(x)$ is continuous on $\mathbb{R}^d$, then Assumption~\ref{asm:level_set} holds for any $\rho_i^0 > 0$.
\end{prop}

\begin{proof}
    For any fixed $\alpha$, the point $x\in S_\alpha$ satisfies
    \begin{equation} \label{eq:lev_bdd}
        \sum_{i=0}^n f_i(x_i) + \sum_{i=0}^{n-1} \frac{\rho_i^0}{2}\norm{x_{i+1} - \varphi(x_i)}^2 \le \alpha.
    \end{equation}
    Since $f_i$ are bounded below, there exists a constant $\beta \in \mathbb{R}$, such that $f_i(x_i) \ge \beta$, for $i=0, \ldots, n$. The inequality~\eqref{eq:lev_bdd} and $\rho_i^0 > 0$ imply that 
    \begin{equation}
        f_0(x_0) \le \alpha - n\beta, 
    \end{equation}
    and 
    \begin{equation}
        \frac{\rho_i^0}{2}\norm{x_{i+1} - \varphi(x_i)}^2 \le \alpha - (n+1)\beta, \text{ for } i=0, \ldots, n-1.
    \end{equation}
    The first inequality shows that there exists $C > 0$, such that $\norm{x_0} < C$. The second shows 
    \begin{equation}
        \norm{x_{i+1}} \le \sqrt{\frac{2}{\rho_i^0}(\alpha - (n+1)\beta)} + \norm{\varphi(x_i)},
    \end{equation}
    for $i=0, \ldots, n-1$. As $\varphi(x)$ is continuous on $\mathbb{R}^d$, it maps bounded sets to bounded sets. Thus, by induction, we can easily show that all $\norm{x_i}$ are bounded.
\end{proof}

We define the local Lipschitz functions and the local $L$-smooth functions as follows.
\begin{defn}
    A function $f$ is said to be local Lipschitz, if for any compact set $S$, there exists a constant $M$ depending on the set $S$, such that
    \begin{equation}
        \norm{f(x) - f(y)} \le M \norm{x - y}, \text{ for any } x, y \in S.
    \end{equation}
    A function $g$ is said to be local $L$-smooth, if for any compact set $S$, there exists a constant $L$ depending on the set S, such that 
    \begin{equation}
        \norm{\nabla g(x) - \nabla g(y)} \le L \norm{x - y}, \text{ for any } x, y \in S.
    \end{equation}
\end{defn}

The next assumption concerns some boundedness and smoothness conditions of the functions $f_i$ for $i=0,\ldots, n$ and $\varphi$ on a bounded set.

\begin{assumption}\label{asm:smoothness}
    Given a compact set $S$, there exist positive constants $M_f, L_f, C_\varphi, M_\varphi, L_\varphi$, such that 
    \begin{align*}
        & \norm{\nabla f_i(x_i)} \le M_f, \ \norm{\nabla f_i(x_i) - \nabla f_i(y_i)} \le L_f\norm{x_i - y_i}, \\
        & \norm{\varphi(x_i)} \le C_\varphi, \ \norm{\nabla \varphi(x_i)} \le M_\varphi, \ \norm{\nabla \varphi(x_i) - \nabla\varphi(y_i)} \le L_\varphi \norm{x_i - y_i},
    \end{align*}
    for all $i=0, \ldots, n$ and for all $x, y \in S$. The constants stated here may depend on the set $S$.
\end{assumption}
Note that this assumption is much weaker than the global $L$-smooth assumptions usually used in the literature, in the sense that if all functions $f_i$ and $\varphi$ are twice continuously differentiable on $\mathbb{R}^d$, then the assumption will be automatically satisfied.

Next, we state the definition of the Karush–Kuhn–Tucker (KKT) point and the theorem on the characterization of the minimal point of the problem~\eqref{eq:optpro}.

\begin{defn}
    The pair $(x^\star, \lambda^\star)$ is said to be a KKT point, if it satisfies 
    \begin{equation}
        \nabla_x L(x^\star, \lambda^\star) = 0 \quad \text{ and } \quad x_{j+1}^\star - \varphi(x_j^\star) = 0,\ \text{ for } j=0, \ldots, n-1,
    \end{equation}
    where the function $L(x, \lambda)$ is the Lagrangian defined by~\eqref{eq:lag}.
\end{defn}

It is easy to verify that the problem~\eqref{eq:optpro} satisfies the well-known linear independence constraint qualification (LICQ) defined as Definition 12.4, Chapter 12.2 in~\cite{nocedal2006numerical}. In addition, by Theorem 12.1, Chapter 12.3 in the same reference, we have the following first-order necessary condition. 
\begin{theorem}
    Suppose that $x^\star = (x_0^\star, \ldots, x_n^\star)$ is a local solution of~\eqref{eq:optpro}, then there is a dual variable $\lambda^\star = (\lambda_0^\star, \ldots, \lambda_{n-1}^\star)$, such that the pair $(x^\star, \lambda^\star)$ is a KKT point.
\end{theorem}
 For further reference, one can also check~\cite{sun2006optimization}. Based on this theorem, although the condition is not sufficient, it guides us to show that the sequences generated by the algorithm~\eqref{eq:admm} converge to some KKT point of the problem. For the characterization of the convergence, we have the following proposition.

\begin{prop}
    Let $(x^k, \lambda^k)$ be generated by the proximal ADMM algorithm~\eqref{eq:admm_ori}, when it is applied to problem~\eqref{eq:optpro}. If $\norm{x^{k+1} - x^k} \to 0$ and $\norm{\lambda^{k+1} - \lambda^k} \to 0$, then all accumulation points, if they exist, are KKT points.
\end{prop}
\begin{proof}
    Let $(x^\star, \lambda^\star)$ be an accumulation point of the sequence $(x^k, \lambda^k)$. Then there is a subsequence $(x^{k_s}, \lambda^{k_s})$ such that $\lim_{s\to \infty} (x^{k_s}, \lambda^{k_s}) \to (x^\star, \lambda^\star)$. The assumption $\norm{x^{k+1} - x^k} \to 0$ and $\norm{\lambda^{k+1} - \lambda^k}\to 0$ as $k\to \infty$ implies $\lim_{s\to \infty} (x^{k_s + 1}, \lambda^{k_s + 1}) \to (x^\star, \lambda^\star)$. The first-order optimality conditions of the subproblems of algorithm~\eqref{eq:admm_ori} is 
    \begin{equation}
        \frac{1}{2\eta_i}\left(x_i^{k_s+1} - x_i^{k_s}\right) = \nabla_i L_\rho(x_{<i}^{k_s+1}, x_i^{k_s+1}, x_{>i}^{k_s}, \lambda^{k_s}),
    \end{equation}
    for $i=0, \ldots, n$, and the dual updates is 
    \begin{equation}
        \frac{1}{\rho_j}\left( \lambda_j^{k_s+1} - \lambda_j^{k_s} \right) = x_{j+1}^{k_s+1} - \varphi(x_j^{k_s+1}),
    \end{equation}
    for $j = 0, \ldots, n-1$. Because the functions $L_\rho(x, \lambda)$ are smooth, by letting $s\to \infty$ we obtain 
    \begin{align} \label{eq:kkt_con}
        \nabla_x L_\rho(x^\star, \lambda^\star) = 0 \quad \text{ and } \quad x_{j+1}^\star - \varphi(x_j^\star) = 0,\ \text{ for } j=0, \ldots, n-1.
    \end{align}
    Notice that we have the relation 
    \begin{equation}
        L(x, \lambda) =  L_\rho(x, \lambda) - \sum_{j=0}^{n-1} \frac{\rho_j}{2} \norm{x_{j+1} - \varphi(x_j)}^2.
    \end{equation}
    Combining this with~\eqref{eq:kkt_con}, we obtain
    \begin{equation}
        \nabla_x L(x^\star, \lambda^\star) = 0 \quad \text{ and } \quad x_{j+1}^\star - \varphi(x_j^\star) = 0,\ \text{ for } j=0, \ldots, n-1,
    \end{equation}
    which implies that $(x^\star, \lambda^\star)$ is a KKT point.
\end{proof}
This proposition indicates that we can characterize the convergence of the proximal ADMM algorithm by estimating the upper bound of the term $\norm{x^{k+1} - x^k}$ and $\norm{\lambda^{k+1} - \lambda^k}$. Our following arguments are based on this idea.

\section{Convergence analysis} \label{sec:con_ana}
In this section, we establish the convergence property of the proximal ADMM algorithm. This framework for convergence analysis can be generalized to other problem formulations. We discuss some possible generalizations of the problem~\eqref{eq:optpro} after the main results. 

\subsection{Main results}\label{sec:main_results}
The main idea in our proof is the Lyapunov analysis. From the iteration~\eqref{eq:admm_ori}, we know that 
\begin{equation}
    L_\rho(x_{<i}^{k+1}, x_i^{k+1}, x_{>i}^k, \lambda^k) + \frac{1}{2\eta_i}\norm{x_i^{k+1} - x_i^k}^2 \le L_\rho(x_{<i}^{k+1}, x_i^k, x_{>i}^k, \lambda^k),
\end{equation}
for $i=0, \ldots, n$. Summing up all the inequalities from $i=0$ to $n$, we obtain 
\begin{equation}\label{eq:primal_dec}
    L_\rho(x^{k+1}, \lambda^k) + \sum_{i=0}^n \frac{1}{2\eta_i} \norm{x_i^{k+1} - x_i^k}^2 \le L_\rho(x^k, \lambda^k).
\end{equation}
In order to analyze the convergence of proximal ADMM, the above inequalities guide us to introduce the Lyapunov function 
\begin{equation}\label{eq:lyapunov}
    E(k) = L_\rho(x^k, \lambda^k) + \sum_{i=0}^n \frac{1}{4\eta_i}\norm{x_i^k - x_i^{k-1}}^2,
\end{equation}
for any $k \ge 1$. Our aim is to show that $E(k)$ decreases as $k$ increases. To achieve that, we calculate 
\begin{align}\label{eq:lya_diff}
    & E(k+1) - E(k) \nonumber \\ 
    ={} & L_\rho(x^{k+1}, \lambda^{k+1}) - L_\rho(x^k, \lambda^k) + \sum_{i=0}^n \frac{1}{4\eta_i}\left(\norm{x_i^{k+1} - x_i^k}^2 - \norm{x_i^k - x_i^{k-1}}^2\right) \nonumber \\
    ={} & L_\rho(x^{k+1}, \lambda^{k+1}) - L_\rho(x^{k+1}, \lambda^k) + L_\rho(x^{k+1}, \lambda^k) - L_\rho(x^k, \lambda^k) \nonumber \\
    & + \sum_{i=0}^n \frac{1}{4\eta_i}\left(\norm{x_i^{k+1} - x_i^k}^2 - \norm{x_i^k - x_i^{k-1}}^2\right) \nonumber \\
    \le & L_\rho(x^{k+1},\lambda^{k+1}) - L_\rho(x^{k+1}, \lambda^k) - \sum_{i=0}^n \frac{1}{4\eta_i}\left(\norm{x_i^{k+1} - x_i^k}^2 + \norm{x_i^k - x_i^{k-1}}^2\right) \nonumber \\
    ={} & \sum_{i=0}^{n-1} \inner{\lambda_i^{k+1} - \lambda_i^{k}, x_{i+1}^{k+1} - \varphi(x_i^{k+1})} - \sum_{i=0}^n \frac{1}{4\eta_i}\left(\norm{x_i^{k+1} - x_i^k}^2 + \norm{x_i^k - x_i^{k-1}}^2\right) \nonumber \\
    ={} & \sum_{i=0}^{n-1} \frac{1}{\rho_i}\norm{\lambda_i^{k+1} - \lambda_i^k}^2 - \sum_{i=0}^n \frac{1}{4\eta_i}\left(\norm{x_i^{k+1} - x_i^k}^2 + \norm{x_i^k - x_i^{k-1}}^2\right),
\end{align}
where the first inequality comes from~\eqref{eq:primal_dec} and the last equality is derived from the update in~\eqref{eq:admm_ori}. The remaining job is to bound the term $\sum_{i=0}^{n-1} \norm{\lambda_i^{k+1} - \lambda_i^k}^2$. Before this, we start from a lemma that provides the bound for $\norm{\lambda^{k+1}}$, which is useful in our later analysis. 
\begin{lemma} \label{lem:dual_bound}
    Let $\lambda^k = (\lambda_0^k, \ldots, \lambda_{n-1}^k)$ be the Lagrange multipliers generated by the algorithm~\eqref{eq:admm_ori}. Suppose that Assumption~\ref{asm:smoothness} holds, and $x^{k+1} = (x_0^{k+1}, \ldots, x_n^{k+1}) \in S$, where $S$ is a bounded set, then 
    \begin{equation}\label{eq:dual_bdd}
        \norm{\lambda_i^{k+1}} \le \frac{1-M_\varphi^{n-i}}{1-M_\varphi}M_f + \sum_{j=i}^{n-2}\rho_{j+1}M_\varphi^{j+1-i}\norm{x_{j+2}^{k+1} - x_{j+2}^k} + \sum_{j=i}^{n-1} \frac{M_\varphi^{j-i}}{\eta_{j+1}}\norm{x_{j+1}^{k+1} - x_{j+1}^k},
    \end{equation}
    for $i = 0, \ldots, n-1$. The bound~\eqref{eq:dual_bdd} can be rewritten as 
    \begin{equation}\label{eq:dual_bound}
        \norm{\lambda_i^{k+1}} \le \frac{1-M_\varphi^{n-i}}{1-M_\varphi}M_f + \sum_{j=i}^{n-1}B_{j,i}\norm{x_{j+1}^{k+1} - x_{j+1}^k},
    \end{equation}
    where 
    \begin{equation}\label{eq:dual_const}
        B_{j,i} = \left\{\begin{aligned}
            & \rho_j M_\varphi^{j-i} + \frac{M_\varphi^{j-i}}{\eta_{j+1}}, & \text{ for } j > i, \\
            & \frac{1}{\eta_{i+1}}, & \text{ for } j=i,
        \end{aligned} \right.
    \end{equation}
    for $i=0, \ldots, n-1$.
\end{lemma}

\begin{proof}
    The proof is an iterative argument on the index $i$. From the first-order optimality condition of the subproblems in the proximal ADMM algorithm~\eqref{eq:admm}, we obtain
    \begin{multline}
        \nabla f_i(x_i^{k+1}) + \rho_{i-1}\left(x_i^{k+1} - \varphi(x_{i-1}^{k+1}) + \frac{1}{\rho_{i-1}}\lambda_{i-1}^k \right) \\- \rho_i \nabla \varphi(x_i^{k+1})^T \left(x_{i+1}^k - \varphi(x_i^{k+1}) + \frac{1}{\rho_i} \lambda_i^k \right) + \frac{1}{\eta_i}\left(x_i^{k+1} - x_i^k\right) = 0,
    \end{multline}
    for $i=1, \ldots, n-1$, and 
    \begin{equation}
        \nabla f_n(x_n^{k+1}) + \rho_{n-1}\left(x_n^{k+1} - \varphi(x_{n-1}^{k+1}) + \frac{1}{\rho_{n-1}}\lambda_{n-1}^k \right) + \frac{1}{\eta_n}\left(x_n^{k+1} - x_n^k\right) = 0.
    \end{equation}
    Combining these equations with the update rule for the multipliers in~\eqref{eq:admm}, the equations are transformed to 
    \begin{equation}\label{eq:first_ord}
        \lambda_{i-1}^{k+1} - \nabla \varphi(x_i^{k+1})^T \lambda_i^{k+1} = -\nabla f_i(x_i^{k+1}) - \rho_i \nabla \varphi(x_i^{k+1})^T \left(x_{i+1}^{k+1} - x_{i+1}^k\right) - \frac{1}{\eta_i}\left(x_i^{k+1} - x_i^k\right),
    \end{equation}
    for $i=1, \ldots, n-1$, and for $i=n$, 
    \begin{equation} \label{eq:fir_ord_n}
        \lambda_{n-1}^{k+1} = -\nabla f_n(x_n^{k+1}) - \frac{1}{\eta_n}\left(x_n^{k+1} - x_n^k\right).
    \end{equation}
    From the last equation, we have that 
    \begin{align}\label{eq:init_est}
        \norm{\lambda_{n-1}^{k+1}} \le & \norm{\nabla f_n(x_n^{k+1})} + \frac{1}{\eta_n} \norm{x_n^{k+1} - x_n^k} \nonumber \\
        \le & M_f + \frac{1}{\eta_n}\norm{x_n^{k+1} - x_n^k}.
    \end{align}
    Next, substituting $i$ to $i+1$ in~\eqref{eq:first_ord}, we have the iterative formula
    \begin{multline}\label{eq:fir_ord}
        \lambda_i^{k+1} = \nabla \varphi(x_{i+1}^{k+1})^T \lambda_{i+1}^{k+1} - \nabla f_{i+1}(x_{i+1}^{k+1}) \\- \rho_{i+1}\nabla \varphi(x_{i+1}^{k+1})^T\left(x_{i+2}^{k+1} - x_{i+2}^k\right) - \frac{1}{\eta_{i+1}}\left(x_{i+1}^{k+1} - x_{i+1}^k\right),
    \end{multline}
    for $i=0, \ldots, n-2$. From these relations, we can deduce by Assumption~\ref{asm:smoothness} that 
    \begin{equation}
        \norm{\lambda_i^{k+1}} \le M_\varphi \norm{\lambda_{i+1}^{k+1}} + M_f + \rho_{i+1}M_\varphi \norm{x_{i+2}^{k+1} - x_{i+2}^k} + \frac{1}{\eta_{i+1}}\norm{x_{i+1}^{k+1} - x_{i+1}^k},
    \end{equation}
    for $i=0, \ldots, n-2$. Multiplying each side by $M_\varphi^i$ and summing up all the inequalities from $i$ to $n-2$, we get
    \begin{align}
        & M_\varphi^i \norm{\lambda_i^{k+1}} \\
        \le{} & M_\varphi^{n-1} \norm{\lambda_{n-1}^{k+1}} + \sum_{j=i}^{n-2}M_f M_\varphi^j + \sum_{j=i}^{n-2}\rho_{j+1}M_\varphi^{j+1} \norm{x_{j+2}^{k+1} - x_{j+2}^k} + \sum_{j=i}^{n-2} \frac{M_\varphi^j}{\eta_{j+1}}\norm{x_{j+1}^{k+1} - x_{j+1}^k} \nonumber \\
        \le{} & \frac{M_\varphi\left(1-M_\varphi^{n-i}\right)}{1-M_\varphi} + \sum_{j=i}^{n-2}\rho_{j+1}M_\varphi^{j+1}\norm{x_{j+2}^{k+1} - x_{j+2}^k} + \sum_{j=i}^{n-1}\frac{M_\varphi^j}{\eta_{j+1}}\norm{x_{j+1}^{k+1} - x_{j+1}^k},
    \end{align}
    where the last inequality comes from the estimate~\eqref{eq:init_est}. Now multiply each side by $M_\varphi^{-i}$, we obtain the estimate for $\norm{\lambda_i^{k+1}}$ as 
    \begin{equation}
        \norm{\lambda_i^{k+1}} \le \frac{1-M_\varphi^{n-i}}{1-M_\varphi} M_f + \sum_{j=i}^{n-2} \rho_{j+1}M_\varphi^{j+1-i}\norm{x_{j+2}^{k+1} - x_{j+2}^k} + \sum_{j=i}^{n-1} \frac{M_\varphi^{j-i}}{\eta_{/j+1}}\norm{x_{j+1}^{k+1} - x_{j+1}^k},
    \end{equation}
    for $i=0, \ldots, n-1$. This is the desired result of this lemma.
\end{proof}

The above lemma enables us to conclude that the bound for $\norm{\lambda_i}$ depends only on $\rho_j$ for $j \ge i+1$, that is, $B_{i, j}$ does not depend on $\rho_i$ for $j=i$. This point will be made clear in later lemmas. 

The nest step is to estimate $\norm{\lambda_i^{k+1} - \lambda_i^k}$. The key idea is the same as in the proof of the above lemma. We will proceed with the estimation by an iterative argument on $i$. We have the following lemma. 
\begin{lemma} \label{lem:dual_diff}
    Let $\lambda^k = (\lambda_0^k, \ldots, \lambda_{n-1}^k)$ be the Lagrange multipliers generated by the proximal ADMM algorithm~\eqref{eq:admm}. Suppose that there exists a bounded set $S$ such that Assumption~\eqref{asm:smoothness} holds and $x^k, x^{k+1} \in S$, then we have
    \begin{equation}
        \norm{\lambda_i^{k+1} - \lambda_i^k}^2 \le \sum_{j=i}^{n-1} 2(n-i)C^2_{j,i} \norm{x_{j+1}^{k+1} - x_{j+1}^k}^2 + \sum_{j=i}^{n-1} 2(n-i)\tilde{C}^2_{j,i}\norm{x_{j+1}^k - x_{j+1}^{k-1}}^2,
    \end{equation}
    where 
    \begin{equation}\label{eq:const_1}
        C_{j,i} = \left\{\begin{aligned}
            & \frac{M_\varphi^{j-i} - M_\varphi^{n-i-1}}{1-M_\varphi} M_f L_\varphi + M_\varphi^{j-i} L_f && \\
            & \phantom{\frac{M_\varphi^{j-i} - M_\varphi^{n-i-1}}{1-M_\varphi} } + (2j - 2i + 1) \frac{M_\varphi^{j-i}}{\eta_{j+1}} + (2j - 2i - 1)\rho_j M_\varphi^{j-i}, && \text{ for } j > i, \\
            & \frac{1 - M_\varphi^{n-i-1}}{1 - M_\varphi} M_f L_\varphi + L_f + \frac{1}{\eta_{i+1}}, && \text{ for } j=i,
        \end{aligned} \right.
    \end{equation}
    and 
    \begin{equation}\label{eq:const_2}
        \tilde{C}_{j,i} = \left\{ \begin{aligned}
            &\frac{M_\varphi^{j-i}}{\eta_{j+1}} + \rho_j M_\varphi^{j-i}, && \text{ for } j > i, \\
            & \frac{1}{\eta_{j+1}}, && \text{ for } j = i,
        \end{aligned}\right.
    \end{equation}
    for $i = 0, \ldots, n-1$.
\end{lemma}

\begin{proof}
    As in the proof of Lemma~\ref{lem:dual_bound}, we have equations~\eqref{eq:fir_ord} and~\eqref{eq:fir_ord_n}. Replacing $k$ by $k-1$ in these equations, we obtain
    \begin{multline}\label{eq:fir_ord_pre}
        \lambda_i^k = \nabla \varphi(x_{i+1}^k)^T \lambda_{i+1}^k - \nabla f_{i+1}(x_{i+1}^k) \\- \rho_{i+1}\nabla \varphi(x_{i+1}^k)^T\left(x_{i+2}^k - x_{i+2}^{k-1}\right) - \frac{1}{\eta_{i+1}}\left(x_{i+1}^k - x_{i+1}^{k-1}\right),
    \end{multline}
    for $i=0, \ldots, n-2$, and 
    \begin{equation} \label{eq:fir_ord_n_pre}
        \lambda_{n-1}^k = -\nabla f_n(x_n^k) - \frac{1}{\eta_n}\left(x_n^k - x_n^{k-1}\right).
    \end{equation}
    Subtracting~\eqref{eq:fir_ord} by~\eqref{eq:fir_ord_pre}, we can show that 
    \begin{align}
        & \lambda_i^{k+1} - \lambda_i^k \nonumber \\
        ={} & \nabla \varphi(x_{i+1}^k)^T \left(\lambda_{i+1}^{k+1} - \lambda_{i+1}^k\right) + \left(\nabla\varphi(x_{i+1}^{k+1}) - \nabla \varphi(x_{i+1}^k) \right)^T \lambda_{i+1}^{k+1} \nonumber \\
        & - \left( \nabla f_{i+1}(x_{i+1}^{k+1}) - \nabla f_{i+1}(x_{i+1}^k) \right) - \rho_{i+1} \nabla\varphi(x_{i+1}^{k+1})^T \left(x_{i+2}^{k+1} - x_{i+2}^k\right) \nonumber \\
        & + \rho_{i+1}\nabla\varphi(x_{i+1}^k)^T \left(x_{i+2}^k - x_{i+2}^{k-1} \right) - \frac{1}{\eta_{i+1}}\left(x_{i+1}^{k+1} - x_{i+1}^k\right) + \frac{1}{\eta_{i+1}}\left( x_{i+1}^k - x_{i+1}^{k-1}\right),
    \end{align}
    for $i=0, \ldots, n-2$. Also subtracting~\eqref{eq:fir_ord_n} by~\eqref{eq:fir_ord_n_pre}, we get 
    \begin{equation}
        \lambda_{n-1}^{k+1} - \lambda_{n-1}^k = -\left(\nabla f_n(x_n^{k+1}) - \nabla f_n(x_n^k) \right) - \frac{1}{\eta_n}\left(x_n^{k+1} - x_n^k\right) + \frac{1}{\eta_n}\left(x_n^k - x_n^{k-1}\right).
    \end{equation}
    From Assumption~\ref{asm:smoothness}, the last equation can be estimated as 
    \begin{align} \label{eq:diff_est_n}
        \norm{\lambda_{n-1}^{k+1} - \lambda_{n-1}^k} \le{} & L_f \norm{x_n^{k+1} - x_n^k} + \frac{1}{\eta_n}\norm{x_n^{k+1} - x_n^k} + \frac{1}{\eta_n}\norm{x_n^k - x_n^{k-1}} \nonumber \\
        ={} & \left( L_f + \frac{1}{\eta_n} \right) \norm{x_n^{k+1} - x_n^k} + \frac{1}{\eta_n} \norm{x_n^k - x_n^{k-1}}.
    \end{align}
    This is the key inequality for the following analysis. From~\eqref{eq:fir_ord_pre}, we obtain
    \begin{align}
        \norm{\lambda_i^{k+1} - \lambda_i^k} \le{} & M_\varphi\norm{\lambda_{i+1}^{k+1} - \lambda_{i+1}^k} + \norm{\nabla\varphi(x_{i+1}^{k+1}) - \nabla\varphi(x_{i+1}^k)}\cdot \norm{\lambda_{i+1}^{k+1}} \nonumber \\
        & + L_f \norm{x_{i+1}^{k+1} - x_{i+1}^k} + \rho_{i+1} M_\varphi \norm{x_{i+2}^{k+1} - x_{i+2}^k} \nonumber \\
        & + \rho_{i+1}M_\varphi \norm{x_{i+2}^k - x_{i+2}^{k-1}} + \frac{1}{\eta_{i+1}}\norm{x_{i+1}^{k+1} - x_{i+1}^k} + \frac{1}{\eta_{i+1}}\norm{x_{i+1}^k - x_{i+1}^{k-1}} \nonumber \\
        ={} & M_\varphi \norm{\lambda_{i+1}^{k+1} - \lambda_{i+1}^k} + \norm{\nabla\varphi(x_{i+1}^{k+1}) - \nabla\varphi(x_{i+1}^k}\cdot \norm{\lambda_{i+1}^{k+1}} \nonumber \\
        & + \left( L_f + \frac{1}{\eta_{i+1}}\right) \norm{x_{i+1}^{k+1} - x_{i+1}^k} + \frac{1}{\eta_{i+1}}\norm{x_{i+1}^k - x_{i+1}^{k-1}} \nonumber \\
        & + \rho_{i+1}M_\varphi \norm{x_{i+2}^{k+1} - x_{i+2}^k} + \rho_{i+1}M_\varphi \norm{x_{i+2}^k - x_{i+2}^{k-1}}
    \end{align}
    where the constants $M_\varphi$ and $L_f$ are the ones in Assumption~\ref{asm:smoothness}. Now, multiplying each side by $M_\varphi^i$, the inequality becomes 
    \begin{align}
        M_\varphi^i \norm{\lambda_i^{k+1} - \lambda_i^k} \le{} & M_\varphi^{i+1}\norm{\lambda_{i+1}^{k+1} - \lambda_{i+1}^k} + M_\varphi^i \norm{\nabla\varphi(x_{i+1}^{k+1}) - \nabla\varphi(x_{i+1}^k)}\cdot \norm{\lambda_{i+1}^{k+1}} \nonumber \\
        & + M_\varphi^i\left(L_f + \frac{1}{\eta_{i+1}}\right) \norm{x_{i+1}^{k+1} - x_{i+1}^k} + \frac{M_\varphi^i}{\eta_{i+1}}\norm{x_{i+1}^k - x_{i+1}^{k-1}} \nonumber \\
        & + \rho_{i+1} M_\varphi^{i+1} \norm{x_{i+2}^{k+1} - x_{i+2}^k} + \rho_{i+1} M_\varphi^{i+1} \norm{x_{i+2}^k - x_{i+2}^{k-1}}.
    \end{align}
    Summing up all the inequalities from $i$ to $n-2$, we get 
    \begin{align}\label{eq:mid_est}
        M_\varphi^i \norm{\lambda_i^{k+1} - \lambda_i^k} \le{} & M_\varphi^{n-1}\norm{\lambda_{n-1}^{k+1} - \lambda_{n-1}^k} + \sum_{j=i}^{n-2}M_\varphi^j \norm{\nabla\varphi(x_{j+1}^{k+1}) - \nabla \varphi(x_{j+1}^k)}\cdot \norm{\lambda_{j+1}^{k+1}} \nonumber \\
        & + \sum_{j=i}^{n-2} M_\varphi^j\left(L_f + \frac{1}{\eta_{j+1}}\right) \norm{x_{j+1}^{k+1} - x_{j+1}^k} + \sum_{j=i}^{n-2}\frac{M_\varphi^j}{\eta_{j+1}}\norm{x_{j+1}^k - x_{j+1}^{k-1}} \nonumber \\
        & + \sum_{j=i}^{n-2} \rho_{j+1}M_\varphi^{j+1}\norm{x_{j+2}^{k+1} - x_{j+2}^k} + \sum_{j=i}^{n-2}\rho_{j+1}M_\varphi^{j+1}\norm{x_{j+2}^k - x_{j+2}^{k-1}} \nonumber \\
        \le{} & \sum_{j=i}^{n-2} M_\varphi^j \norm{\nabla \varphi(x_{j+1}^{k+1}) - \nabla\varphi(x_{j+1}^k)}\cdot \norm{\lambda_{j+1}^{k+1}} \nonumber \\
        & + \sum_{j=i}^{n-1} M_\varphi^j\left(L_f + \frac{1}{\eta_{j+1}}\right) \norm{x_{j+1}^{k+1} - x_{j+1}^k} + \sum_{j-i}^{n-1}\frac{M_\varphi^j}{\eta_{j+1}}\norm{x_{j+1}^k - x_{j+1}^{k-1}} \nonumber \\
        & + \sum_{j=i}^{n-2} \rho_{j+1}M_\varphi^{j+1} \norm{x_{j+2}^{k+1} - x_{j+2}^k} + \sum_{j=i}^{n-2} \rho_{j+1}M_\varphi^{j+1}\norm{x_{j+2}^k - x_{j+2}^{k-1}},
    \end{align}
    where the last inequality is due to~\eqref{eq:diff_est_n}. In order to continue our estimation, an interlude for the estimate of the term $\sum_{j=i}^{n-2} M_\varphi^j \norm{\nabla \varphi(x_{j+1}^{k+1}) - \nabla\varphi(x_{j+1}^k)}\cdot \norm{\lambda_{j+1}^{k+1}}$ is needed. From Lemma~\ref{lem:dual_bound}, we have that 
    \begin{align}\label{eq:middle_est}
        & \sum_{j=i}^{n-2} M_\varphi^j \norm{\nabla\varphi(x_{j+1}^{k+1}) - \nabla\varphi(x_{j+1}^k)}\cdot \norm{\lambda_{j+1}^{k+1}} \nonumber \\
        \le{} & \sum_{j=i}^{n-2} M_\varphi^j \norm{\nabla\varphi(x_{j+1}^{k+1}) - \nabla\varphi(x_{j+1}^k)} \cdot \left( \frac{1-M_\varphi^{n-j-1}}{1-M_\varphi} M_f + \sum_{q=j+1}^{n-2} \rho_{q+1} M_\varphi^{q-j}\norm{x_{q+2}^{k+1} - x_{q+2}^k} \right. \nonumber\\
        & \left. \phantom{\sum_{j=i}^{n-2} M_\varphi^j \norm{\nabla\varphi(x_{j+1}^{k+1}) - \nabla\varphi(x_{j+1}^k)} \cdot ( \frac{1-M_\varphi^{n-j-1}}{1-M_\varphi} M_f} + \sum_{q=j+1}^{n-1} \frac{M_\varphi^{q-j-1}}{\eta_{q+1}}\norm{x_{q+1}^{k+1} - x_{q+1}^k} \right) \nonumber \\
        \le{} & \sum_{j=i}^{n-2}M_\varphi^j \frac{1 - M_\varphi^{n-j-1}}{1-M_\varphi} M_f L_\varphi \norm{x_{j+1}^{k+1} - x_{j+1}^k} + \sum_{j=i}^{n-2} M_\varphi^j\cdot 2M_\varphi \sum_{q=j+1}^{n-2}\rho_{q+1}M_\varphi^{q-j}\norm{x_{q+2}^{k+1} - x_{q+2}^k} \nonumber \\
        & + \sum_{j=i}^{n-2} M_\varphi^j \cdot 2M_\varphi \sum_{q=j+1}^{n-1} \frac{M_\varphi^{q-j-1}}{\eta_{q+1}}\norm{x_{q+1}^{k+1} - x_{q+1}^k} \nonumber \\
        ={} & \sum_{j=i}^{n-2} \frac{M_\varphi^j - M_\varphi^{n-1}}{1-M_\varphi} M_f L_\varphi \norm{x_{j+1}^{k+1} - x_{j+1}^k} + 2\sum_{q=i+1}^{n-2}\sum_{j=i}^{q-1}\rho_{q+1} M_\varphi^{q+1} \norm{x_{q+2}^{k+1} - x_{q+2}^k} \nonumber \\
        & + 2 \sum_{q=i+1}^{n-1} \sum_{j=i}^{q-1} \frac{M_\varphi^q}{\eta_{q+1}} \norm{x_{q+1}^{k+1} - x_{q+1}^k} \nonumber \\
        ={} & \sum_{j=i}^{n-2} \frac{M_\varphi^j - M_\varphi^{n-1}}{1-M_\varphi} M_f L_\varphi \norm{x_{j+1}^{k+1} - x_{j+1}^k} + \sum_{q=i+1}^{n-2}2(q-i) \rho_{q+1}M_\varphi^{q+1}\norm{x_{q+2}^{k+1} - x_{q+2}^k} \nonumber \\
        & + \sum_{q=i+1}^{n-1} 2(q-i) \frac{M_\varphi^q}{\eta_{q+1}} \norm{x_{q+1}^{k+1} - x_{q+1}^k}.
    \end{align}
    The second inequality comes from Assumption~\ref{asm:smoothness} and the estimates 
    \begin{equation}
        \norm{\nabla\varphi(x_{q+1}^{k+1}) - \nabla\varphi(x_{q+1}^k)} \le L_\varphi \norm{x_{q+1}^{k+1} - x_{q+1}^k}, \text{ and } \norm{\nabla\varphi(x_{q+1}^{k+1}) - \nabla\varphi(x_{q+1}^k)} \le 2 M_\varphi.
    \end{equation}
    Now, we return to the estimation~\eqref{eq:mid_est}. Combining the estimate~\eqref{eq:middle_est} with~\eqref{eq:mid_est}, we obtain 
    \begin{align}
        & M_\varphi^i \norm{\lambda_i^{k+1} - \lambda_i^k} \nonumber \\
        \le{} & \sum_{j=i}^{n-2} \frac{M_\varphi^j - M_\varphi^{n-1}}{1 - M_\varphi} M_f L_\varphi \norm{x_{j+1}^{k+1} - x_{j+1}^k} + \sum_{j=i+1}^{n-2} 2(j-i) \rho_{j+1}M_\varphi^{j+1}\norm{x_{j+2}^{k+1} - x_{j+2}^k} \nonumber \\
        & + \sum_{j=i+1}^{n-1} 2(j-i) \frac{M_\varphi^j}{\eta_{j+1}}\norm{x_{j+1}^{k+1} - x_{j+1}^k} + \sum_{j=i}^{n-1} M_\varphi^j \left(L_f + \frac{1}{\eta_{j+1}}\right) \norm{x_{j+1}^{k+1} - x_{j+1}^k} \nonumber \\
        & + \sum_{j=i}^{n-1}\frac{M_\varphi^j}{\eta_{j+1}}\norm{x_{j+1}^k - x_{j+1}^{k-1}} + \sum_{j=i}^{n-2}\rho_{j+1}M_\varphi^{j+1}\norm{x_{j+2}^{k+1} - x_{j+2}^k} + \sum_{j=i}^{n-2}\rho_{j+1} M_\varphi^{j+1}\norm{x_{j+2}^k - x_{j+2}^{k-1}} \nonumber \\
        ={} & \sum_{j=i}^{n-1} \left( \frac{M_\varphi^j - M_\varphi^{n-1}}{1-M-\varphi} M_f L_\varphi + M_\varphi^j L_f + (2j - 2i + 1)\frac{M_\varphi^j}{\eta_{j+1}} \right) \norm{x_{j+1}^{k+1} - x_{j+1}^k} \nonumber \\
        & + \sum_{j=i}^{n-1}\frac{M_\varphi^j}{\eta_{j+1}}\norm{x_{j+1}^k - x_{j+1}^{k-1}} + \sum_{j=i}^{n-2}(2j - 2i + 1)\rho_{j+1}M_\varphi^{j+1}\norm{x_{j+2}^{k+1} - x_{j+2}^k} \nonumber \\
        & + \sum_{j=i}^{n-2}\rho_{j+1}M_\varphi^{j+1}\norm{x_{j+2}^k - x_{j+2}^{k-1}}.
    \end{align}
    Finally, multiplying each side by $M_\varphi^{-i}$, we have that 
    \begin{align}
        & \norm{\lambda_i^{k+1} - \lambda_i^k} \nonumber \\
        \le{} & \sum_{j=i}^{n-1}\left( \frac{M_\varphi^{j-i} - M_\varphi^{n-i-1}}{1 - M_\varphi} M_f L_\varphi + M_\varphi^{j-i} L_f + (2j - 2i + 1)\frac{M_\varphi^{j-i}}{\eta_{j+1}} \right) \norm{x_{j+1}^{k+1} - x_{j+1}^k} \nonumber \\
        & + \sum_{j=i}^{n-1} \frac{M_\varphi^{j-i}}{\eta_{j+1}}\norm{x_{j+1}^k - x_{j+1}^{k-1}} + \sum_{j=i}^{n-2}(2j - 2i + 1) \rho_{j+1} M_\varphi^{j-i+1} \norm{x_{j+2}^{k+1} - x_{j+2}^k} \nonumber \\
        & + \sum_{j=i}^{n-2} \rho_{j+1} M_\varphi^{j-i+1}\norm{x_{j+2}^k - x_{j+2}^{k-1}} \nonumber \\
        ={} & \sum_{j=i+1}^{n-1} \Bigg( \frac{M_\varphi^{j-i} - M_\varphi^{n-i-1}}{1-M-\varphi} M_f L_\varphi + M_\varphi^{j-i} L_f \nonumber \\
        & \phantom{\sum_{j=i+1}^{n-1} \Bigg( \frac{M_\varphi^{j-i} - M_\varphi^{n-i-1}}{1-M-\varphi} M_f} + (2j - 2i + 1)\frac{M_\varphi^{j-i}}{\eta_{j+1}} + (2j - 2i - 1)\rho_j M_\varphi^{j-i} \Bigg) \norm{x_{j+1}^{k+1} - x_{j+1}^k} \nonumber \\
        & + \left( \frac{1 - M_\varphi^{n-i-1}}{1 - M_\varphi} M_f L_\varphi + L_f + \frac{1}{\eta_{i+1}} \right) \norm{x_{i+1}^{k+1} - x_{i+1}^k} \nonumber \\
        & + \sum_{j=i+1}^{n-1} \left( \frac{M_\varphi^{j-i}}{\eta_{j+1}} + \rho_j M_\varphi^{j-i} \right) \norm{x_{j+1}^k - x_{j+1}^{k-1}} + \frac{1}{\eta_{i+1}}\norm{x_{i+1}^k - x_{i+1}{k-1}} \nonumber \\
        ={} & \sum_{j=i}^{n-1} C_{j,i} \norm{x_{j+1}^{k+1} - x_{j+1}^k} + \sum_{j=i}^{n-1} \tilde{C}_{j,i}\norm{x_{j+1}^k - x_{j+1}^{k-1}}
    \end{align}
    for $i = 0, \ldots, n-1$. Here the constants $C_{j,i}$ and $C_{j,i}$ are defined by~\eqref{eq:const_1} and~\eqref{eq:const_2}. Squaring both sides of the above inequality and using the fact $\left( \sum_{i=1}^n a_i \right)^2 \le n \left(\sum_{i=1}^n a_i^2 \right)$, we obtain the desired result.
\end{proof}

The above lemma is the key idea in our proof of the convergence of the proximal ADMM algorithm~\eqref{eq:admm}, because it provides an upper bound for the change of the multipliers in terms of the changes of the variables, which in turn enables us to assign the parameters properly so that the Lyapunov function~\eqref{eq:lyapunov} decreases for each iteration $k$. From Lemma~\ref{lem:dual_diff}, we have that 
\begin{align}
    & \sum_{i=0}^{n-1}\frac{1}{\rho_i}\norm{\lambda_i^{k+1} - \lambda_i^k}^2 \nonumber \\
    \le{} & \sum_{i=0}^{n-1}\frac{1}{\rho_i}\left( \sum_{j=i}^{n-1} 2(n-i)C_{j,i}^2\norm{x_{j+1}^{k+1} - x_{j+1}^k}^2 + 2(n-i)\tilde{C}_{j,i}^2\norm{x_{j+1}^k - x_{j+1}^{k-1}}^2 \right) \nonumber \\
    ={} & \sum_{j=0}^{n-1}\left( \sum_{i=0}^j \frac{2(n-i)}{\rho_i}C_{j,i}^2 \right) \norm{x_{j+1}^{k+1} - x_{j+1}^k}^2 + \left( \sum_{i=0}^j \frac{2(n-i)}{\rho_i}\tilde{C}_{j,i}^2 \right) \norm{x_{j+1}^k - x_{j+1}^{k-1}}^2 \nonumber \\
    ={} & \sum_{i=1}^n \left(\sum_{j=0}^{i-1} \frac{2(n-j)}{\rho_j}C_{i-1, j}^2 \right) \norm{x_i^{k+1} - x_i^k}^2 + \left(\sum_{j=0}^{i-1} \frac{2(n-j)}{\rho_j}\tilde{C}_{i-1,j}^2 \right) \norm{x_i^k - x_i^{k-1}}^2.
\end{align}
Substituting this inequality into~\eqref{eq:lya_diff}, the change in the Lyapunov function for one step is upper bounded by 
\begin{align} \label{eq:lya_diff_fin}
    E(k+1) - E(k) \le{} & \sum_{i=1}^n \left(\sum_{j=0}^{i-1} \frac{2(n-j)}{\rho_j}C_{i-1, j}^2  - \frac{1}{4\eta_i}\right) \norm{x_i^{k+1} - x_i^k}^2 \nonumber \\
    & + \sum_{i=1}^n \left(\sum_{j=0}^{i-1} \frac{2(n-j)}{\rho_j}\tilde{C}_{i-1,j}^2 - \frac{1}{4\eta_i} \right) \norm{x_i^k - x_i^{k-1}}^2.
\end{align}
In order to have the decreasing property of the Lyapunov function, we need the coefficients of the right-hand side to be negative. Because of the truth $0 < \tilde{C}_{j,i} \le C_{j,i}$, we only need 
\begin{equation}\label{eq:para_con}
    \sum_{j=0}^{i-1}\frac{2(n-j)}{\rho_j}C_{i-1, j}^2 < \frac{1}{4\eta_i}
\end{equation}
to hold, for $i=1, \ldots, n$. The fact that the constant $C_{j,i}$ defined by~\eqref{eq:const_1} does not depend on $\rho_j$ for $j=i$ plays a vital role, which enables us to set the parameters successively to satisfy the above condition~\eqref{eq:para_con}. 

\textbf{A procedure to choose the parameters.} The program to set the parameters is as follows: First, we can choose arbitrarily $\eta_i$, for $i=0,\ldots, n$, then consider the diagram 
\begin{equation}
    \begin{matrix}
        \frac{2n}{\rho_0}C_{0,0}^2 & & & & &\\
        \frac{2n}{\rho_0}C_{1,0}^2 & & \frac{2(n-1)}{\rho_1}C_{1, 1}^2 & & &\\
        \vdots & & \cdots & \ddots & &\\
        \frac{2n}{\rho_0}C_{n-2, 0}^2 & & \cdots &  & \frac{4}{\rho_{n-2}}C_{n-2,n-2}^2 & \\
        \frac{2n}{\rho_0}C_{n-1, 0}^2 & & \cdots & & \frac{4}{\rho_{n-2}}C_{n-1, n-2}^2 & \frac{2}{\rho_{n-1}}C_{n-1,n-1}^2
    \end{matrix}
\end{equation}
By definition~\eqref{eq:const_1}, the quantity $C_{j, i}$ depends only on $\rho_j$ for $i < j$ and is constant for $i = j$. Our aim is to set the parameters $\rho_j$ for $j=0, \ldots, n-1$ such that the sum of the $i$-th line is less than $\frac{1}{4\eta_i}$. We start from the bottom line. First, we set $\rho_{n-1}$ large enough such that $\frac{2}{\rho_{n-1}}C^2_{n-1,n-1} < \frac{1}{n}\cdot \frac{1}{4\eta_n}$ and then fix $\rho_{n-1}$. Now the remaining $C_{n-1,0}, \ldots, C_{n-1,n-2}$ becomes fixed. We set $\rho_0, \ldots, \rho_{n-2}$ large enough such that $\sum_{j=0}^{n-2} \frac{2(n-j)}{\rho_j} C^2_{n-1, j} < \frac{n-1}{n}\cdot \frac{1}{4\eta_n}$. In this way, we obtain parameters $\rho_0, \ldots, \rho_{n-1}$ such that the sum of the bottom line is less than $\frac{1}{4\eta_n}$. We fix the parameter $\rho_{n-1}$ and go to the second line from the bottom. We proceed in the same way: first, choose $\rho_{n-2}$ such that $\frac{4}{\rho_{n-2}}C^2_{n-2, n-2} < \frac{1}{n-1}\cdot \frac{1}{4\eta_{n-1}}$ and then choose $\rho_0, \ldots, \rho_{n-3}$ such that $\sum_{j=0}^{n-3} \frac{2(n-j)}{\rho_j}C^2_{n-2, j} < \frac{n-2}{n-1}\cdot \frac{1}{4\eta_{n-1}}$ but notice that when we choose the parameters $\rho_0, \ldots, \rho_{n-2}$ we also require them to be larger than the ones we have from the last procedure. Then the sum of the bottom line is even smaller, and thus still satisfies the desired inequality. Now we have a fixed $\rho_{n-2}$. By successively performing this procedure, we fixed all parameters $\rho_j$ for $j=0, \ldots, n-1$ and the inequalities~\eqref{eq:para_con} are satisfied.

The only thing left to show is that the sequence of primal variables generated by the proximal ADMM algorithm stay in a bounded set, which is the following lemma.

\begin{lemma} \label{lem:lya_des}
    Let $x^k = (x_0^k, \ldots, x_n^k)$ and $\lambda^k = (\lambda_0^k, \lambda_{n-1}^k)$ be the primal and dual variables generated by the proximal ADMM algorithm~\eqref{eq:admm}. Suppose that the functions $f_i$ and $\varphi$ satisfy Assumption~\ref{asm:level_set} and there is a constant $M>0$, such that the initial point $x^0$ satisfies $\norm{x_{i+1}^0 - \varphi(x_i^0)}^2 \le M/\rho_i$, for $i=0,\ldots,n-1$. Define 
    \begin{equation}
        \bar{\alpha} = \sum_{i=0}^n f_i(x_i^0) + \sum_{i=0}^{n-1} \frac{5\norm{\lambda_i^0}^2}{4\rho_i^0} + n(M+3).
    \end{equation}Suppose that Assumption~\ref{asm:smoothness} holds with the level set $S_{\bar{\alpha}}$ defined by~\eqref{eq:level_set}. In addition, assume that the parameters satisfy~\eqref{eq:para_con} and 
    \begin{equation}\label{eq:assum}
        \rho_i > \max\left\{2\rho_i^0,\ \left(\frac{1-M_\varphi^{n-i}}{1-M_\varphi}M_f + \sum_{j=i}^{n-1}B_{j,i} D_{\bar{\alpha},j+1} \right)^2\right\},
    \end{equation}
    for $i=0,\ldots, n-1$, where $\rho_i^0$ is the constant in Assumption~\ref{asm:level_set} and $B_{j,i}$ is defined by~\eqref{eq:dual_const}. The constants $D_{\bar{\alpha},i}$ are defined as 
    \begin{equation}\label{eq:def_bound}
        D_{\bar{\alpha},i} := \max\left\{ \norm{x_i - y_i} \Big|\ x, y \in S_{\bar{\alpha}},\ x=(x_0,\ldots,x_n),\ y=(y_0, \ldots, y_n) \right\}.
    \end{equation} 
    Then the sequence 
    \begin{equation}
        x^k \in S_{\bar{\alpha}} \quad\text{ and }\quad \norm{\lambda_i^k} \le \frac{1-M_\varphi^{n-i}}{1-M_\varphi} + \sum_{j=i}^{n-1}B_{j,i}D_{\bar{\alpha},j+1}, \text{ for } i=0,\ldots, n-1,
    \end{equation}
    for all $k \ge 1$. Furthermore, the Lyapunov function is monotone decreasing and the following inequality holds for any $k \ge 1$.
    \begin{equation}
        E(k+1) - E(k) \le - \sum_{i=0}^n \left(c_i \norm{x_i^{k+1} - x_i^k}^2 + \tilde{c}_i \norm{x_i^k - x_i^{k-1}}^2 \right),
    \end{equation}
    where 
    \begin{equation} \label{eq:const_lya}
        c_i = \frac{1}{4\eta_i} - \sum_{j=0}^{i-1} \frac{2(n-j)}{\rho_j}C_{i-1, j}^2 \quad \text{ and } \quad \tilde{c}_i = \frac{1}{4\eta_i} - \sum_{j=0}^{i-1} \frac{2(n-j)}{\rho_j}\tilde{C}_{i-1, j}^2.
    \end{equation}
\end{lemma}

\begin{proof}
    We prove the lemma by induction. First, we investigate the case where $k = 0$ and $k = 1$. Note that 
    \begin{align}\label{eq:est_0}
        L_\rho(x^0, \lambda^0) ={} & \sum_{i=0}^n f_i(x_i^0) + \sum_{i=0}^{n-1}\inner{\lambda_i^0, x_{i+1}^0 - \varphi(x_i^0)} + \frac{\rho_i}{2}\norm{x_{i+1}^0 - \varphi(x_i^0)}^2 \nonumber \\
        \le{} & \sum_{i=0}^n f_i(x_i^0) + \sum_{i=0}^{n-1} \frac{\norm{\lambda_i^0}^2}{2\rho_i} + \rho_i \norm{x_{i+1}^0 - \varphi(x_i^0)}^2 \nonumber \\
        \le{} & \sum_{i=0}^n f_i(x_i^0) + \sum_{i=0}^{n-1} \frac{\norm{\lambda_i^0}^2}{4\rho_i^0} + nM,
    \end{align}
    where the first inequality is because $\inner{a, b} \le \frac{1}{2\rho}\norm{a}^2 + \frac{\rho}{2}\norm{b}^2$ and the last inequality comes from $\rho_i > 2\rho_i^0$ and $\norm{x^0_{i+1} - \varphi(x_i^0)}^2 \le M/\rho_i$. From the primal descent property~\eqref{eq:primal_dec}, we have that 
    \begin{equation}
        L_\rho(x^1, \lambda^0) \le L_\rho(x^0, \lambda^0),
    \end{equation}
    which implies 
    \begin{align}
        & \sum_{i=0}^n f_i(x_i^1) + \sum_{i=0}^{n-1} \frac{\rho_i^0}{2}\norm{x_{i+1}^1 - \varphi(x_i^1)}^2\nonumber \\
        \le{} & L_\rho(x^0, \lambda^0) - \sum_{i=0}^{n-1} \inner{\lambda_i^0, x_{i+1}^1 - \varphi(x_i^1)} - \sum_{i=0}^{n-1} \frac{\rho_i - \rho_i^0}{2}\norm{x_{i+1}^1 - \varphi(x_i^1)}^2  \nonumber \\
        \le{} & L_\rho(x^0, \lambda^0) + \sum_{i=0}^{n-1}\frac{\norm{\lambda_i^0}^2}{2(\rho_i - \rho_i^0)} \nonumber \\
        \le{} & L_\rho(x^0, \lambda^0) + \sum_{i=0}^{n-1} \frac{\norm{\lambda_i^0}^2}{2\rho_i^0} \nonumber \\
        \le{} & \sum_{i=0}^n f_i(x_i^0) + \sum_{i=0}^{n-1} \frac{3\norm{\lambda_i^0}^2}{4\rho_i^0} + nM \le \bar{\alpha}.
    \end{align}
    The last inequality is given by~\eqref{eq:est_0}. By Assumption~\ref{asm:level_set}, we have that $x^1 \in S_{\bar{\alpha}}$. In addition, 
    \begin{align}
        \sum_{i=0}^n f_i(x_i^0) + \sum_{i=0}^{n-1} \frac{\rho_i^0}{2}\norm{x_{i+1}^0 - \varphi(x_i^1)}^2 \le{} & \sum_{i=0}^n f_i(x_i^1) + \sum_{i=0}^{n-1} \frac{\rho_i^0 M}{2 \rho_i} \nonumber \\
        \le{} & \sum_{i=0}^n f_i(x_i^1) + \frac{n}{4}M \le \bar{\alpha}.
    \end{align}
    This implies $x^0\in S_{\bar{\alpha}}$. For the estimation of dual variables, from~\eqref{eq:dual_bound} and the assumption~\eqref{eq:assum} we have that
    \begin{align}\label{eq:est_lam_1}
        \norm{\lambda_i^1}^2 \le{} & \left(\frac{1-M_\varphi^{n-i}}{1-M_\varphi} M_f + \sum_{j-i}^{n-1}B_{j,i}\norm{x_{j+1}^1 - x_{j+1}^0}\right)^2 \nonumber \\
        \le{} & \left(\frac{1-M_\varphi^{n-i}}{1-M_\varphi} M_f + \sum_{j-i}^{n-1}B_{j,i}D_{\bar{\alpha},j+1}\right)^2 \le \rho_i,
    \end{align}
    for $i=0, \ldots, n-1$. In order to perform induction in later argument, we also need the estimate for $E(1)$ as 
    \begin{align}
        E(1) ={} & L_\rho(x^1, \lambda^1) + \sum_{i=0}^n \frac{1}{4\eta_i}\norm{x_i^1 - x_i^0}^2 \nonumber \\
        ={} & L_\rho(x^1, \lambda^1) - L_\rho(x^1, \lambda^0) + L_\rho(x^1, \lambda^0) - L_\rho(x^0, \lambda^0) \nonumber \\
        & + L_\rho(x^0,\lambda^0) + \sum_{i=0}^n \frac{1}{4\eta_i}\norm{x_i^1 - x_i^0}^2 \nonumber \\
        \le{} & \sum_{i=0}^{n-1} \frac{1}{\rho_i}\norm{\lambda_i^1 - \lambda_i^0}^2 + L_\rho(x^0, \lambda^0) - \sum_{i=0}^n \frac{1}{4\eta_i}\norm{x_i^1 - x_i^0}^2 \nonumber \\
        \le{} & \sum_{i=0}^{n-1} \left(\frac{2}{\rho_i}\norm{\lambda_i^1}^2 + \frac{2}{\rho_i}\norm{\lambda_i^0}^2\right) + \sum_{i=0}^n f_i(x_i^0) + \sum_{i=0}^{n-1} \frac{\norm{\lambda_i^0}^2}{4\rho_i^0} + nM \nonumber \\
        \le{} & \sum_{i=0}^n f_i(x_i^0) + \sum_{i=0}^{n-1} \frac{5\norm{\lambda_i^0}^2}{4\rho_i^0} + n(M+2).
    \end{align}
    where the first inequality is the primal descent and dual ascent. The second inequality comes from $\norm{a + b}^2 \le 2\norm{a}^2 + 2\norm{b}^2$ and the estimate~\eqref{eq:est_0}. For the last inequality, we use the condition $\rho_i > \rho_i^0$ and~\eqref{eq:est_lam_1}. 

    Now, we prove the estimates by induction. Suppose that the sequences generated by the proximal ADMM~\eqref{eq:admm} satisfy 
    \begin{subequations}
    \begin{align}
        & x^k \in S_{\bar{\alpha}}, \\
        & \norm{\lambda_i^k} \le \frac{1-M_\varphi^{n-i}}{1-M_\varphi} + \sum_{j=i}^{n-1}B_{j,i}D_{\bar{\alpha},j+1}, \text{ for } i=0,\ldots, n-1, \\
        & E(k) \le \sum_{i=0}^n f_i(x_i^0) + \sum_{i=0}^{n-1} \frac{5\norm{\lambda_i^0}^2}{4\rho_i^0} + n(M+2), \label{eq:energy_est}
    \end{align}
    \end{subequations}
    for some $k \ge 1$, then we show that these properties also hold for $k+1$. Note that by the primal descent condition, we have that 
    \begin{equation}
        L_\rho(x^{k+1}, \lambda^k) \le L_\rho(x^k, \lambda^k) \le E(k),
    \end{equation}
    which yields 
    \begin{equation}
        \sum_{i=0}^n f_i(x_i^{k+1}) + \sum_{i=0}^{n-1} \inner{\lambda_i^k, x_{i+1}^{k+1} - \varphi(x_i^{k+1})} + \frac{\rho_i}{2}\norm{x_{i+1}^{k+1} - \varphi(x_i^{k+1})}^2 \le E(k).
    \end{equation}
    With some manipulation, we obtain that 
    \begin{multline}
        \sum_{i=0}^n f_i(x_i^{k+1}) + \sum_{i=0}^{n-1} \frac{\rho_i^0}{2}\norm{x_{i+1}^{k+1} - \varphi(x_i^{k+1})}^2 \\ + \sum_{i=0}^{n-1} \frac{\rho_i - \rho_i^0}{2}\norm{x_{i+1}^{k+1} - \varphi(x_i^{k+1}) + \frac{1}{\rho_i - \rho_i^0}\lambda_i^k}^2 - \sum_{i=0}^{n-1} \frac{\norm{\lambda_i^k}^2}{2(\rho_i - \rho_i^0)} \le E(k).
    \end{multline}
    We deduce that 
    \begin{align}\label{eq:est_lev_bdd}
        \sum_{i=0}^n f_i(x_i^{k+1}) + \sum_{i=0}^{n-1} \frac{\rho_i^0}{2}\norm{x_{i+1}^{k+1} - \varphi(x_i^{k+1})}^2 \le{} & E(k) + \sum_{i=0}^{n-1} \frac{\norm{\lambda_i^k}^2}{2(\rho_i - \rho_i^0)} \nonumber \\
        \le{} & E(k) + \sum_{i=0}^{n-1} \frac{\rho_i}{2(\rho_i - \rho_i^0)} \nonumber \\
        \le{} & E(k) + n.
    \end{align}
    The second inequality comes from 
    \begin{equation}
        \norm{\lambda_i^k}^2 \le \left(\frac{1-M_\varphi^{n-i}}{1-M_\varphi} + \sum_{j=i}^{n-1}B_{j,i}D_{\bar{\alpha},j+1}\right) \le \rho_i, \text{ for } i=0,\ldots, n-1,
    \end{equation}
    and the last inequality holds because the function $\rho_i/(2(\rho_i - \rho_i^0))$ is monotonically decreasing for $\rho_i \ge 2\rho_i^0$. Combining~\eqref{eq:est_lev_bdd} with~\eqref{eq:energy_est}, we get 
    \begin{equation}
        \sum_{i=0}^n f_i(x_i^{k+1}) + \sum_{i=0}^{n-1} \frac{\rho_i^0}{2}\norm{x_{i+1}^{k+1} - \varphi(x_i^{k+1})}^2 \le \sum_{i=0}^n f_i(x_i^0) + \sum_{i=0}^{n-1} \frac{5\norm{\lambda_i^0}^2}{4\rho_i^0} + n(M+3) = \bar{\alpha}.
    \end{equation}
    This implies that $x^{k+1} \in S_{\bar{\alpha}}$. Then from Lemma~\eqref{lem:dual_bound}, equation~\eqref{eq:dual_bound} and the definition of $D_{\bar{\alpha}, i}$~\eqref{eq:def_bound}, we obtain 
    \begin{align}
        \norm{\lambda_i^{k+1}} \le \frac{1-M_\varphi^{n-i}}{1-M_\varphi} + \sum_{j=i}^{n-1}B_{j,i}D_{\bar{\alpha},j+1}, 
    \end{align}
    for $i=0, \ldots, n-1$. As $x^k, x^{k+1} \in S_{\bar{\alpha}}$, by Lemma~\eqref{lem:dual_diff} and~\eqref{eq:lya_diff_fin}, we have that 
    \begin{equation} \label{eq:dec_lya}
        E(k+1) - E(k) \le - \sum_{i=0}^n \left(c_i \norm{x_i^{k+1} - x_i^k}^2 + \tilde{c}_i \norm{x_i^k - x_i^{k-1}}^2 \right),
    \end{equation}
    where $c_i$ and $\tilde{c}_i$ are defined by~\eqref{eq:const_lya}. By assumption, $c_i > 0$ and $\tilde{c}_i > 0$ for $i=0, \ldots, n$, thus 
    \begin{equation}
        E(k+1) \le E(k) \le \sum_{i=0}^n f_i(x_i^0) + \sum_{i=0}^{n-1} \frac{5\norm{\lambda_i^0}^2}{4\rho_i^0} + n(M+2).
    \end{equation}
    This completes our proof.
\end{proof}

Based on the previous lemmas, the convergence properties of the sequence generated by the algorithm~\eqref{eq:admm} are clear. If we choose a suitable initial point and set the appropriate values for the parameters, the sequences will stay in a bounded domain (thus in a compact domain) and the Lyapunov function decreases as~\eqref{eq:dec_lya}. We can then estimate the convergence rate of the algorithm~\eqref{eq:admm} and have the following theorem.

\begin{theorem}
    With the same assumptions as in Lemma~\eqref{lem:lya_des}, the sequence $x^k$ generated by the proximal ADMM~\eqref{eq:admm} satisfies $\norm{x^{k+1} - x^k} \to 0$ and $\norm{\lambda^{k+1} - \lambda^k} \to 0$ as $k \to \infty$. Furthermore, we have 
    \begin{equation}\label{eq:fin_est_ine}
        \min_{j=1, \ldots, k} \sum_{i=0}^n  \left(c_i + \tilde{c}_i\right) \norm{x_i^{j+1} - x_i^j}^2 \le o(1/k),
    \end{equation}
    where the constant $c_i$ and $\tilde{c}_i$ are defined by~\eqref{eq:const_lya}. Inequality~\eqref{eq:fin_est_ine} implies that among all the subsequences that converge, the fastest converges at a rate of $o(1/k)$.
\end{theorem}
\begin{proof}
    First, we show that the Lyapunov function~\eqref{eq:lyapunov} is lower-bounded. Note that 
    \begin{align}
        L_\rho(x^k, \lambda^k) ={} & \sum_{i=0}^n f_i(x_i^k) + \sum_{i=0}^{n-1} \inner{\lambda_i^k, x_{i+1}^k - \varphi(x_i^k)} + \sum_{i=0}^{n-1} \frac{\rho_i}{2} \norm{x_{i+1}^k - \varphi(x_i^k)}^2 \nonumber \\
        ={} & \sum_{i=0}^n f_i(x_i^k) + \sum_{i=0}^{n-1} \frac{\rho_i^0}{2}\norm{x_{i+1}^k - \varphi(x_i^k)}^2 \nonumber \\
        & + \sum_{i=0}^{n-1}\frac{\rho_i - \rho_i^0}{2}\norm{x_{i+1}^k - \varphi(x_i^k) + \frac{1}{\rho_i - \rho_i^0} \lambda_i^k}^2 - \sum_{i=0}^{n-1}\frac{\norm{\lambda_i^k}^2}{2(\rho_i - \rho_i^0)}.
    \end{align}
    From the assumption and Lemma~\ref{lem:lya_des}, we know that $\rho_i \ge 2\rho_i^0$ and $\norm{\lambda_i^k}^2 \le \rho_i$. So, 
    \begin{align}
        L_\rho(x^k, \lambda^k) \ge{} & \sum_{i=0}^n f_i(x_i^k) + \sum_{i=0}^{n-1} \frac{\rho_i^0}{2}\norm{x_{i+1}^k - \varphi(x_i^k)}^2 - \sum_{i=0}^{n-1} \frac{\rho_i}{2(\rho_i - \rho_i^0)} \nonumber \\
        \ge{} & B_{\bar{\alpha}} - n.
    \end{align}
    Hence, the Lyapunov function is bounded below by 
    \begin{equation}
        E(k) \ge L(x^k, \lambda^k) \ge B_{\bar{\alpha}} - n,
    \end{equation}
    for all $k \in \mathbb{N}^+$. Now, because inequality~\eqref{eq:dec_lya} holds, summing up all the inequalities from $k=1$ to $k=N$ yields that 
    \begin{equation}
        E(N+1) - E(1) \le -\sum_{k=1}^N \sum_{i=0}^n  \left( c_i \norm{x_i^{k+1} - x_i^k}^2 + \tilde{c}_i \norm{x_i^k - x_i^{k-1}}^2 \right).
    \end{equation}
    By some rearrangement, we have 
    \begin{equation}
        \sum_{k=1}^N \sum_{i=0}^n  \left( c_i \norm{x_i^{k+1} - x_i^k}^2 + \tilde{c}_i \norm{x_i^k - x_i^{k-1}}^2 \right) \le E(1) - E(N+1) \le E(1) - B_{\bar{\alpha}} + n.
    \end{equation}
    Let $N\to \infty$, and we obtain
    \begin{equation}\label{eq:infinte_series}
        \sum_{k=1}^\infty \sum_{i=0}^n \left(c_i + \tilde{c}_i\right)\norm{x_i^{k+1} - x_i^k}^2 \le E(1) - B_{\bar{\alpha}} + n.
    \end{equation}
    This immediately gives us $\norm{x^{k+1} - x^k} \to 0$ as $k\to \infty$. From Lemma~\ref{lem:dual_diff} we know that $\norm{\lambda^{k+1} - \lambda^k} \to 0$ as $k\to \infty$. 
    
    Furthermore, we have that
    \begin{align}\label{eq:final_est}
        \min_{j=1, \ldots, k} \sum_{i=0}^n  \left(c_i + \tilde{c}_i\right) \norm{x_i^{j+1} - x_i^j}^2 \le{} & \min_{j=\lfloor k/2 \rfloor, \ldots, k} \sum_{i=0}^n  \left(c_i + \tilde{c}_i\right) \norm{x_i^{j+1} - x_i^j}^2 \nonumber \\
        \le{} & \frac{2}{k+4}\sum_{j=\lfloor k/2\rfloor}^k \sum_{i=0}^n  \left( c_i \norm{x_i^{j+1} - x_i^j}^2 + \tilde{c}_i \norm{x_i^j - x_i^{j-1}}^2 \right).
    \end{align}
    Because the infinite series in~\eqref{eq:infinte_series} converges, the sum 
    \begin{equation}
        \sum_{j=\lfloor k/2\rfloor}^k \sum_{i=0}^n  \left( c_i \norm{x_i^{j+1} - x_i^j}^2 + \tilde{c}_i \norm{x_i^j - x_i^{j-1}}^2 \right)
    \end{equation}
    goes to $0$ as $k\to \infty$. This, together with~\eqref{eq:final_est}, implies the $o(1/k)$ convergence rate for the minimum of $\sum_{i=0}^n  \left(c_i + \tilde{c}_i\right) \norm{x_i^{j+1} - x_i^j}^2$.
\end{proof}

\subsection{Generalization and discussion} \label{sec:generalization}
Our analysis in Section~\ref{sec:main_results} provides a framework for analyzing the convergence of the proximal ADMM method~\eqref{eq:admm_ori} applied to more general dynamics-constrained optimization problems. In this subsection, we discuss several possible generalizations of the problem and the result we can have.

\textbf{Non-autonomous dynamics and semi-implicit schemes.}
The dynamics as constraints in~\eqref{eq:con_constraints} is called non-autonomous if the function $\Phi = \Phi(x,t)$ depends on time $t$. For the numerical schemes, in order to reduce the associated stability constraint, the principle linear operators are treated implicitly, while the nonlinear terms are treated explicitly to avoid expensive computation of solving the nonlinear equations. This kind of discretized problems can be formulated as 
\begin{subequations}
    \begin{align}
        \min_{x}\ & \sum_{i=0}^n f_i(x_i), \\
        \mathrm{s.\ t.}\ & A_{j}x_{j+1} = \varphi_j(x_j), \ j=0,\ldots, n-1,
    \end{align}
\end{subequations}
where $A_j\in \mathbb{R}^{d\times d}$ are invertible matrices. The algorithm~\eqref{eq:admm_ori} applied to this problem is also convergent, as the same analysis framework in Section~\ref{sec:main_results} works.

\textbf{Implicit schemes.}
We can also consider the implicit scheme to numerically discretize problem~\eqref{eq:con_constraints}. In this case, the optimization problem becomes 
\begin{subequations}\label{eq:imp_pro}
    \begin{align}
        \min_{x}\ & \sum_{i=0}^n f_i(x_i), \\
        \mathrm{s.\ t.}\ & \phi(x_{j+1}) = x_j, \ j=0, \ldots, n-1.
    \end{align}
\end{subequations}
If we directly apply algorithm~\eqref{eq:admm_ori} to this problem, the proof in Section~\ref{sec:main_results} does not work without further assumptions on the gradient of the function $\phi$. But notice that, with the change of variables $y_i = x_{n-i}$ for $i=0,\ldots, n$, the formulation of problem~\eqref{eq:imp_pro} becomes the same as problem~\eqref{eq:optpro}. This suggests that we need to reverse the order when updating the subproblems in algorithm~\eqref{eq:admm_ori}, that is, the following method will have the same convergence properties as the algorithm~\eqref{eq:admm_ori} for problem~\eqref{eq:optpro}.
\begin{equation}\label{eq:imp_admm}
\begin{cases}
    x_i^{k+1} = \argmin \left\{ L_\rho(x_{<i}^{k+1}, x_i, x_{>i}^k, \lambda^k) + \frac{1}{2\eta_i}\norm{x_i - x_i^k}^2 \right\}, & \text{ for } i=n, \ldots, 0,\\
    \lambda_j^{k+1} = \lambda_j^k + \rho_j\left( \phi(x_{j+1}^{k+1}) - x_j^{k+1} \right), & \text{ for } j=0, \ldots, n-1.
\end{cases}
\end{equation}
The order in the primal update matters in the convergence proof, but for the dual variables, the order of updating does not influence the convergence.

\textbf{Optimal control problems.}
The optimal control problems usually involve finding an optimal control of some dynamics that minimizes the cost function. In the continuous form, it is the following variational problem 
\begin{subequations}
\begin{align}
    \min_{U(t)\in \mathcal{A}}\ & \int_0^T F(X(t), U(t), t) \dd t + G(X(T)), \\
    \mathrm{s.\ t.}\ & \dot{X}(t) = \Phi(X(t), U(t)) \\
    & X(0) = x_0.
\end{align}
\end{subequations}
where $\mathcal{A}$ is the function space for all admissible controls. $x^0$ is the initial condition of the dynamical system. Considering an explicit scheme for this variational problem, we can get 
\begin{subequations}\label{eq:opt_con}
    \begin{align}
        \min_{u\in \mathcal{A}_h}\ & \sum_{i=0}^n f_i(x_i, u_i), \\
        \mathrm{s.\ t.}\ & x_{j+1} = \varphi(x_j, u_j),\ j=0, \ldots, n-1,
    \end{align}
\end{subequations}
where $x = (x_1, \ldots, x_n) \in \mathbb{R}^{nd}$ and $u = (u_0, \ldots, u_n) \in \mathbb{R}^{(n+1)d}$. $x_0$ is the given initial condition. The set $\mathcal{A}_h$ consists of all discrete admissible controls. The augmented Lagrangian for problem~\eqref{eq:opt_con} is 
\begin{equation}
    L_\rho(x, u, \lambda) = \sum_{i=0}^n f_i(x_i, u_i) + \sum_{i=0}^{n-1}\inner{\lambda_i, x_{i+1} - \varphi(x_i, u_i)} + \sum_{i=0}^{n-1} \frac{\rho_i}{2}\norm{x_{i+1} - \varphi(x_i, u_i)}^2.
\end{equation}
We consider the algorithm 
\begin{equation} \label{eq:admm_optcon}
    \left\{ \begin{aligned}
        & u_q^{k+1} = \argmin_{u_q \in \mathcal{A}_h} \left\{ L_\rho(x^k, u_{<q}^{k+1}, u_q, u_{>q}, \lambda^k) + \frac{1}{2\xi_q}\norm{u_q - u_q^k}^2 \right\}, & \text{ for } q = 0, \ldots, n, \\
        & x_i^{k+1} = \argmin_{x_i} \left\{ L_\rho(x_{<i}^{k+1}, x_i, x_{>i}^k, u^{k+1}, \lambda^k) + \frac{1}{2\eta_i}\norm{x_i - x_i^k}^2 \right\}, & \text{ for } i=1, \ldots, n, \\
        & \lambda_j^{k+1} = \lambda_j^k + \rho_j \left( x_{j+1}^{k+1} - \varphi(x_j^{k+1}, u_j^{k+1}) \right), & \text{ for } j = 0, \ldots, n-1.
    \end{aligned}\right.
\end{equation}
The update order for the variables $u_q$ and $\lambda_j$ can be arbitrary. However, we should update $u$ before updating $x$. In addition, when solving the subproblems for $x_i$, we need to compute $x_0^{k+1}$ to $x_n^{k+1}$ successively. With the framework proposed in Section~\ref{sec:main_results}, we can have the convergence guarantee for this algorithm~\eqref{eq:admm_optcon}.

Further generalization such as non-autonomous systems, semi-implicit schemes, and implicit schemes can be similarly obtained as in the previous discussion.

\section{Numerical experiments} \label{sec:num_exp}

In this section, numerical experiments are performed for the proximal ADMM. We discuss the way to solve the subproblems and possible derivative-free implementations. We test the algorithm on 4D variational data assimilation problems and implement it as the solver for implicit schemes of stiff problems.

\subsection{4DVar of Lorenz system}

Let us consider the classical Lorenz system, a well-known example of a three-dimensional, aperiodic, and nonlinear deterministic system known for its chaotic behavior, first studied by~\cite{lorenz1963deterministic}.  It is described by the following system of differential equations:
\begin{equation}
\label{lorenz}
\left\{ \begin{aligned}
         & \dot{x} = \sigma (y - x) \\
         & \dot{y} = x(\rho - z) - y \\
         & \dot{z} = xy - \beta z, 
         \end{aligned} \right.
\end{equation}
where the parameters are set to the classical values: $\sigma = 10$, $\rho=28$ and $\beta = 8/3$. To be compatible with the former notation, we set $v = (x, y, z)^T$. To numerically solve the Lorenz system~\eqref{lorenz}, we employ the fourth order Runge–Kutta method with a time step size of $\delta t = 0.01$ over a total time horizon of $T=3$, thus in total there are $n=300$ numerical steps. We write for simplicity the discrete dynamics after applying the numerical scheme by 

\begin{equation}
    v_{j+1} = \varphi(v_{j}), \text{ for } j = 0, \ldots, n-1,
\end{equation} 
where 
\begin{equation}
    \varphi(v_j) = v_j + \frac{\delta t}{6}(k_1 + 2 k_2 + 2 k_3 + k_4)
\end{equation}
with
\begin{align*}
		k_1 = f(v_j),\ k_2 = f(v_j + \frac{\delta t}{2}k_1),\ k_3 = f(v_j + \frac{\delta t}{2}k_2),\ k_4 = f(v_j + \delta t \cdot k_3).
	\end{align*}
The 4DVar is an optimization problem that aims to recover the dynamics via noised observation~\cite{le1986variational, courtier1987variational, talagrand1987variational}. It can be stated as the following problem.
\begin{subequations}\label{eq:constrained_opt}
\begin{align}
    \min_{v}\ & \sum_{i=0}^n f_i(v_i), \\
    \mathrm{s.t.}\ & v_{j+1} = \varphi(v_j), \ j = 0, \ldots, n-1,
\end{align}
\end{subequations}
where $v = (v_0, \ldots, v_n)$ and 
\begin{equation}
    f_i(v_i) =  \begin{cases}
        \frac{1}{2}\norm{v_0 - \hat{v}_0}^2 + \frac{\alpha}{2}\norm{v_0 - \hat{v}_b}^2, & \text{ if } i=0, \\
        \frac{1}{2}\norm{v_i - \hat{v}_{i/M}}^2, &\text{ if } i/M \in \mathbb{N}^+, \\
        0, & \text{ otherwise}.
    \end{cases}
\end{equation}
The $\hat{v}_i$ for $i=0,\ldots, n$ are the given observations and the $\hat{v}_b$ is the an estimate of the initial value. We set $\hat{v}_b = \hat{v}_0$ for simplicity and set $\alpha = 0.1$. Here we generate the observations by running the discrete dynamics starting from $(-0.5, 0.5, 20.5)$ to generate the true dynamics and then add a standard Guassian noise to each observation. We take the observation every $30$ steps, i.e., we set $M = 30$, thus there are $11$ observations. The way we create observation can be written as 
\begin{equation}
    \hat{v}_{j} = \bar{v}_{M\cdot j} + \epsilon_j, \text{ for } j=0, \ldots, n/M,
\end{equation}
where $\bar{v}_j$ is the true dynamics and $\epsilon_j \sim \mathcal{N}(0, \mathrm{Id}_3)$ are i.i.d random variables. 

\textbf{Adjoint method:} A traditional method to solve this problem is the adjoint method. Note that the constrained problem~\eqref{eq:constrained_opt} can be transformed into an unconstrained problem. 
\begin{equation} \label{eq:unconstrainted_opt}
    \min_{v_0} \sum_{i=0}^n f_i(v_i),
\end{equation}
where $v_i = \varphi^i(v_0)$. The adjoint method is a way to compute the gradient with respect to $v_0$. The procedure is the following. Given a point $v_0$, first we compute $v_i$ for all $i=1, \ldots, n$ by the iteration $v_{i+1} = \varphi(v_i)$. Then we solve the backward adjoin equation 
\begin{equation}
    \begin{cases}
        w_{j-1} = \nabla \varphi(v_{j-1})^T w_j + \nabla f_j(v_j), \\
        w_n = \nabla f_n(v_n).
    \end{cases}
\end{equation}
It is easy to show that the gradient of~\eqref{eq:unconstrainted_opt} is $w_0$. Once the gradient is computed, we can apply the CG and quasi-Newton method to solve this nonconvex optimization problem. Here we set the initial point to be the noise observation $\hat{v}_0$ and implement the Polak–Ribi\`ere CG and L-BFGS-B method. The Figure~\ref{fig:cg_bfgs} shows the numerical results.
\begin{figure}[htbp]
\centering
\begin{subfigure}[t]{0.45\linewidth}
\centering
\includegraphics[scale=0.22]{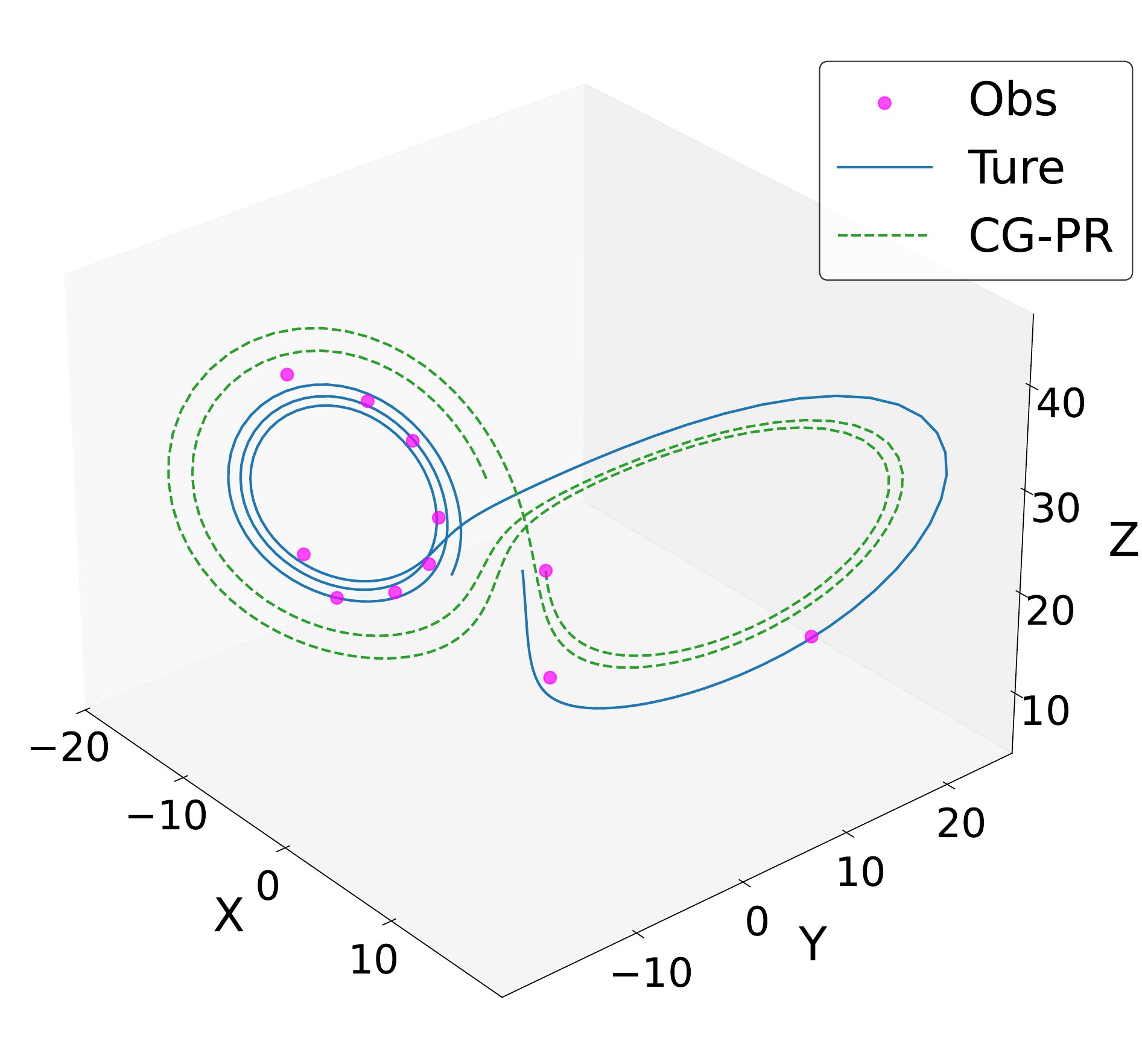}
\caption{Conjugate gradient method}
\end{subfigure}
\begin{subfigure}[t]{0.45\linewidth}
\centering
\includegraphics[scale=0.22]{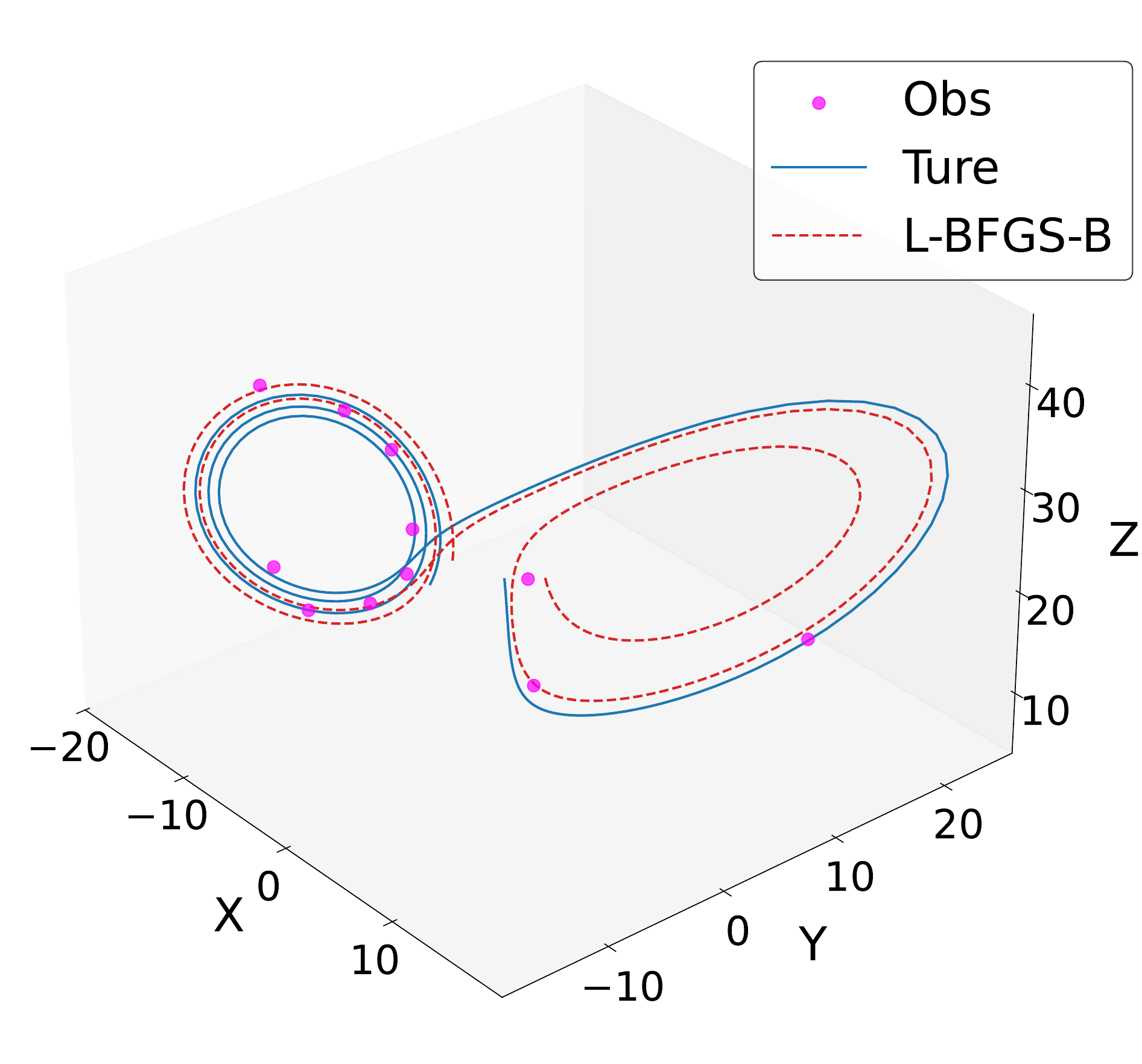}
\caption{Limited memory BFGS-B method}
\end{subfigure}
\caption{\small The Polak–Ribi\`ere CG and L-BFGS-B method are applied to the 4DVar problem~\eqref{eq:unconstrainted_opt} for the Lorenz system~\eqref{lorenz}, starting with the initial condition $\hat{v}_0$.} 
\label{fig:cg_bfgs}
\end{figure}

Note that both methods do not converge to the true dynamics, because the noise of the initial point is too large, and the objective function is highly nonconvex as demonstrated in~\cite{li2025numerical}.

\textbf{Proximal ADMM:} The difference of the proximal ADMM algorithm is that the initial value is able to use all observations, which could cancel the effect of noise and enables the recovery of the true dynamics. The parameters $\rho_j$ and $\eta_j$ can influence the KKT point to which the algorithm converges. If $\rho_j$ is too small, the algorithm will not converge, but if it is too large, the algorithm will be trapped at the local minimum near the initial point. We set $\rho_j = 0.3$ for all $j = 0, \ldots, n-1$ and $\eta_i = 10$ for all $i = 0, \ldots, n$. For the initial value, we first set the initial value $v_{M\cdot j}^0 = \hat{v}_j$, then solve all the points in between using equation $v_{j+1} = \varphi(v_j)$ and set them as the initial value $v_j^0$. We set $\lambda^0_j = (0, \ldots, 0)$ for all $j = 0, \ldots, n-1$. The subproblems in the proximal ADMM algorithm are nonlinear least squares optimization problems. There are many developed method to solving these subproblems such as Levenberg–Marquardt algorithm~\cite{fan2005quadratic}. For the initial value, we can choose the point $x_i^k$ from the previous iteration. As the parameter $\eta_i$ controls the weight of the proximal term in the subproblem, the smaller $\eta_i$ is, the easier it is to solve the subproblem. For low-dimensional problems, we can even use derivative-free methods as subproblem solvers. For the Lorenz system, we implement the derivative-free method, and show its performance.

\begin{figure}[htbp]
\centering
\begin{subfigure}[t]{0.32\linewidth}
\centering
\includegraphics[scale=0.14]{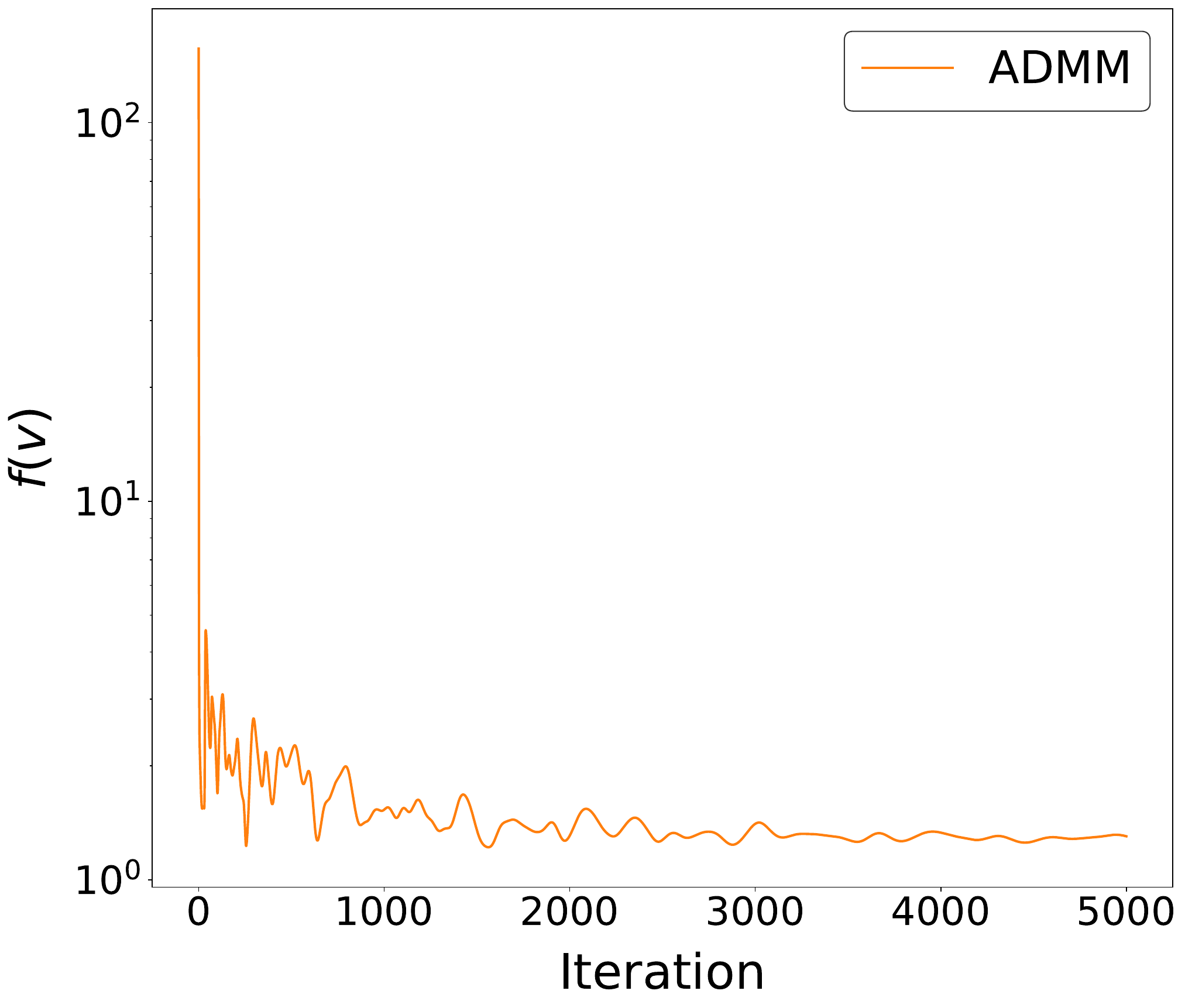}
\caption{\small Function Value}
\label{subfig:objective}
\end{subfigure}
\begin{subfigure}[t]{0.32\linewidth}
\centering
\includegraphics[scale=0.14]{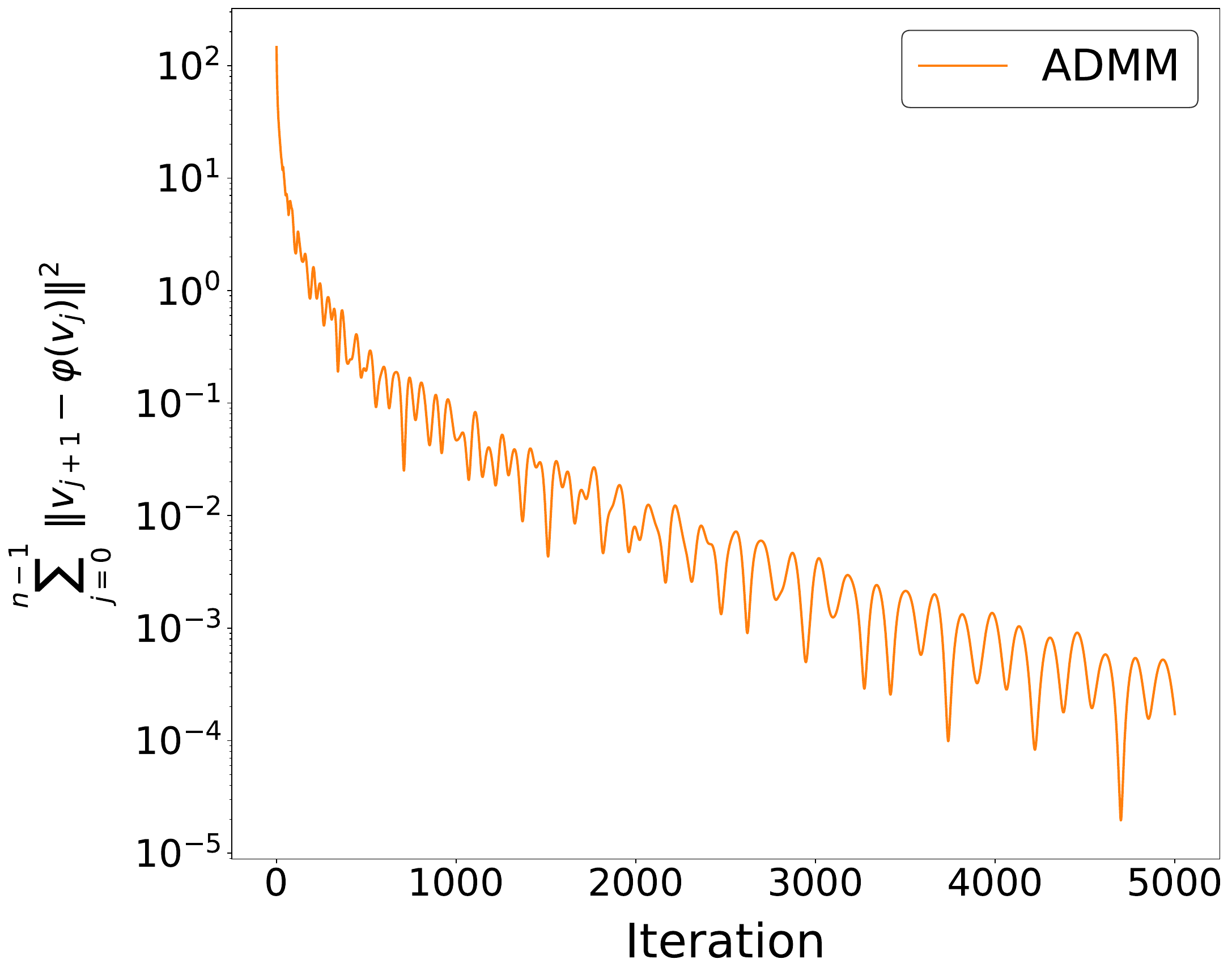}
\caption{\small Constraint error}
\label{subfig:cons_error}
\end{subfigure}
\begin{subfigure}[t]{0.32\linewidth}
\centering
\includegraphics[scale=0.14]{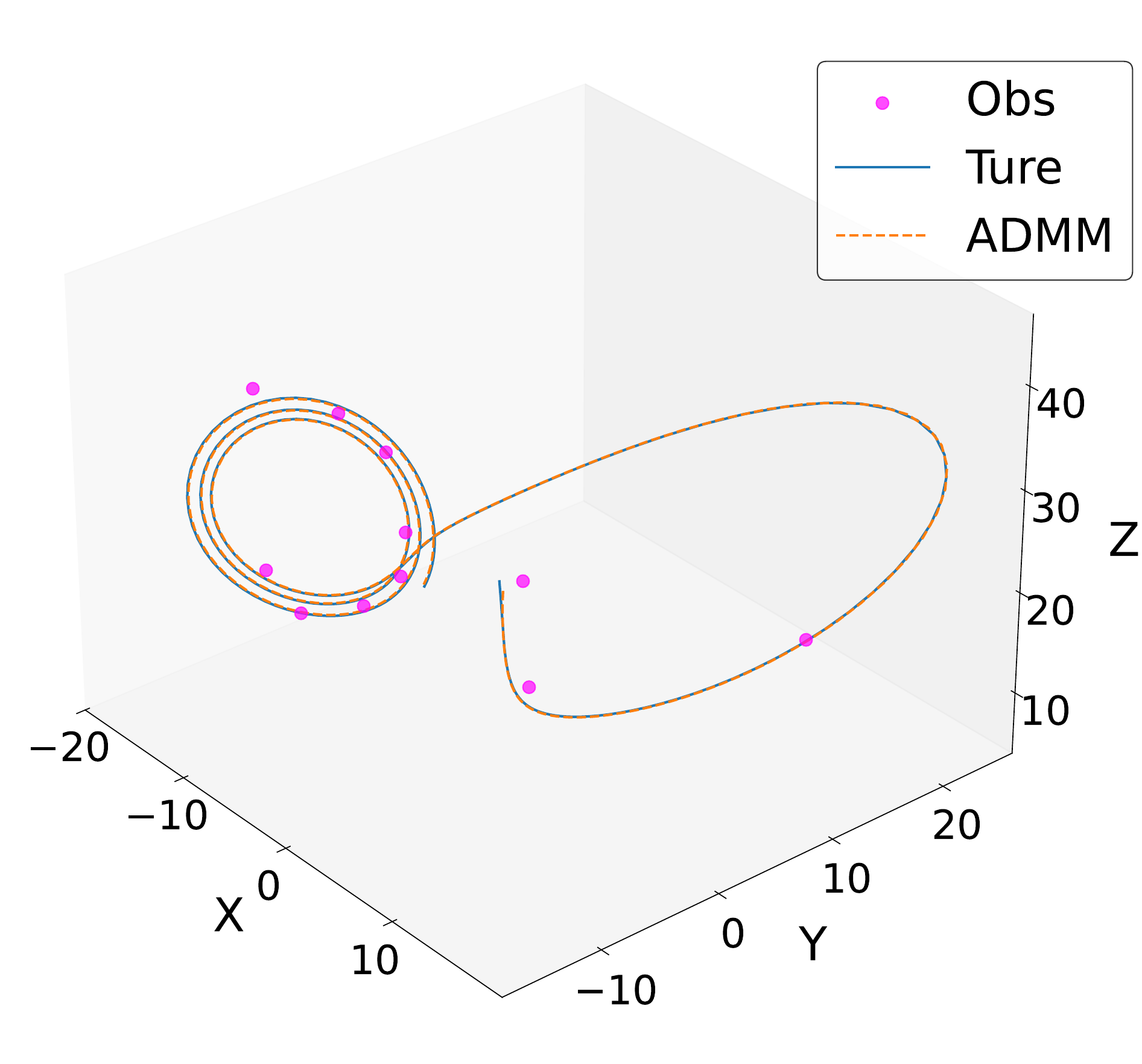}
\caption{\small Recovered solution}
\label{subfig:rec_solu}
\end{subfigure}
\caption{\small Numerical performance of the proximal ADMM~\eqref{eq:admm} applied to the 4DVar problem~\eqref{eq:constrained_opt} for the Lorenz system~\eqref{lorenz}.} 
\label{fig:lorenz_admm}
\end{figure}

Figure~\ref{fig:lorenz_admm} shows the numerical results of the proximal ADMM. The algorithm converges to the true dynamics, which prevails over the former gradient-based methods. This algorithm is not a feasible method. We also plot the figure of the constraint error. It is clear that the constraint error converges to zero as the iteration increases. In addition, it is worth mentioning that the value of objective function usually converges to its minimum in the very beginning, and then it stays the same while the constraint error converges. This phenomenon has been observed in~\cite{li2025numerical} as well. Aside from the advantages, we also need to point out that the speed of convergence of the proximal ADMM now is still very slow compared to the gradient-based methods, which usually have a super linear convergence rate. This leaves out a valuable question of how to accelerate the convergence of ADMM-type methods for future research.

\subsection{Implicit schemes of Burgers' equation}

The stiff equation is an important class of problems we face when considering numerical methods for differential equations. There is no clear definition of stiffness in the literature. Roughly speaking, for some evolution equations, certain numerical schemes, especially the explicit schemes, need the time step to be small enough to have a stable solution and obtain an acceptable accuracy. The stiff equation arises in many areas, such as modeling chemical reactions, control theory, electrical circuits, and relaxation oscillations. Most time-dependent PDEs produce a stiff ODE system after space discretization. Usually, the solution is using the implicit schemes or developing certain stable semi-implicit schemes to overcome the stiffness. In this section, we show the ability of the proximal ADMM algorithm to solve implicit schemes as discussed in Section~\ref{sec:generalization}, which is different from the traditional method of solving nonlinear equations. We use one-dimensional viscous Burgers' equation as a demonstration.

We consider the viscous Burgers' equation with Dirichlet boundary conditions.
\begin{equation}
    \begin{cases}
        \partial_t u + u\cdot \partial_x u = \nu \partial_{xx} u, & x\in (0,\pi),\ t\in (0, T], \\
        u(x, 0) = \sin(x), & x\in [0, \pi]\\
        u(0, t) = u(\pi, t) = 0, & t\in (0, T]
    \end{cases}
\end{equation}
where $\nu$ is the viscosity coefficient. For our numerical experiment, we set $\nu = 0.005$ and $T = 2$. Suppose that the solution satisfies $u\in L^2([0,T); H_0^1([0,\pi]))$ and $\partial_t u \in L^2([0,T]; H^{-1}([0,\pi]))$. 

For the numerical method of the viscous Burgers' equation, let $\delta t$ be the time step size, and the number of spatial grid is set to $m=100$ i.e., the dimension $d=101$, providing a spatial discretization size of $\delta x = \pi/m$. Thus, the number of time steps is $n = T/\delta t$. For $i=0, \ldots, n$, let $u_{i,q}$ be the numerical approximation of $u(i\delta t, q\delta x)$.

For the explicit scheme, we apply the following Lax-Friedrichs scheme.
\begin{align}
    u_{i+1, q} ={} & \frac{1}{2}(u_{i, q+1} + u_{i, q-1}) - \frac{\delta t}{2 \delta x}\left( \frac{u_{i,q+1}^2}{2} - \frac{u_{i, q-1}^2}{2}\right) + \nu\frac{\delta t}{\delta x^2} (u_{i, q+1} - 2u_{i, q} + u_{i, q-1}) \nonumber \\
    ={} & u_{i, q} - \frac{\delta t}{\delta x} (\hat{f}_{i, q+\frac{1}{2}} - \hat{f}_{i, q-\frac{1}{2}}) + \nu\frac{\delta t}{\delta x^2} (u_{i, q+1} - 2u_{i, q} + u_{i, q-1}),
\end{align}
where 
\begin{equation}
    \hat{f}_{i,q+\frac{1}{2}} = \frac{1}{2}\left( \frac{u_{i, q}^2}{2} + \frac{u_{i, q+1}^2}{2} - \frac{\delta x}{\delta t}(u_{i, q+1} - u_{i, q})\right)
\end{equation}
is the Lax-Friedrichs numerical flux. Let $u_i = (u_{i,0}, \ldots, u_{i,m}) \in \mathbb{R}^d$. With a little abuse of notation, we set $u = (u_0, \ldots, u_n) \in \mathbb{R}^{(n+1)d}$. Thus, the explicit scheme can be formulated as $u_{i+1} = \varphi(u_i)$.

For the implicit scheme, we use the implicit Lax-Friedrichs scheme, which is 
\begin{align}
    u_{i+1, q} ={} & u_{i, q} - \frac{\delta t}{\delta x} (\hat{f}_{i+1, q+\frac{1}{2}} - \hat{f}_{i+1, q-\frac{1}{2}}) + \nu\frac{\delta t}{\delta x^2} (u_{i+1, q+1} - 2u_{i+1, q} + u_{i+1, q-1}) \nonumber \\
    ={} & u_{i,q} + \frac{1}{2}(u_{i+1, q+1} - 2u_{i+1,q} + u_{i+1, q-1}) - \frac{\delta t}{2 \delta x}\left( \frac{u_{i+1,q+1}^2}{2} - \frac{u_{i+1, q-1}^2}{2}\right) \nonumber \\
    & + \nu\frac{\delta t}{\delta x^2} (u_{i+1, q+1} - 2u_{i+1, q} + u_{i+1, q-1}).
\end{align}
This implicit scheme can be written in the form $\phi(u_{i+1}) = u_i$. Suppose that we have the initial value $\hat{u}_0$, we can iteratively calculate the value $u_i$ for $i \ge 1$ using the explicit scheme $u_{i+1} = \varphi(u_i)$, but for the implicit scheme, the solution does not come from a straightforward iteration computation. A traditional method is to consider the following nonlinear equation.
\begin{equation}
    \begin{pmatrix}
        \phi(u_1) \\
        \phi(u_2) - u_1 \\
        \vdots \\
        \phi(u_n) - u_{n-1}
    \end{pmatrix} = 
    \begin{pmatrix}
        \hat{u}_0 \\
        0 \\
        \vdots \\
        0
    \end{pmatrix}.
\end{equation}
Then apply some solver for nonlinear equations such as Newton method to get the solution for the implicit scheme. Here, instead, we formulate the implicit scheme as the following optimization problem. 
\begin{subequations}\label{eq:imp_solver}
\begin{align}
    \min\ & \frac{1}{2}\norm{u_0 - \hat{u}_0}^2, \\
    \mathrm{s.t.}\ & \phi(u_{i+1}) = u_i,\ i=0, \ldots, n-1,
\end{align}
\end{subequations}
where it is obvious that the solution to this problem is the solution of this implicit scheme. However, this problem is exactly in the form of problem~\eqref{eq:imp_pro} with $f_0(u_0) = \frac{1}{2}\norm{u_0 - \hat{u}_0}^2$ and $f_i(u_i) = 0$ for $i\ge 1$. We can then apply the proximal ADMM algorithm~\eqref{eq:imp_admm} to the problem~\eqref{eq:imp_solver}. The advantage of this optimization formulation is that it enables further consideration of the design of algorithms that can be implemented parallel in time, for example, the algorithm discussed in~\cite{li2025numerical}. For the parameters, we set $\rho_j = 0.1$ for all $j=0, \ldots, n-1$ and $\eta_i = 2$ for all $i=0, \ldots, n$. The initial points of the iteration are $u_i = (0, \ldots, 0)$ for all $i=0, \ldots, n$ and $\lambda_j^0 = (0, \ldots, 0)$ for all $j = 0, \ldots, n-1$. We perform experiments for $\delta t = 0.1,\ n=20$ and $\delta t = 0.02,\ n=100$, respectively, to exhibit the stability of the implicit scheme for stiff problems.
\begin{figure}[htbp]
\centering
\includegraphics[scale=0.27]{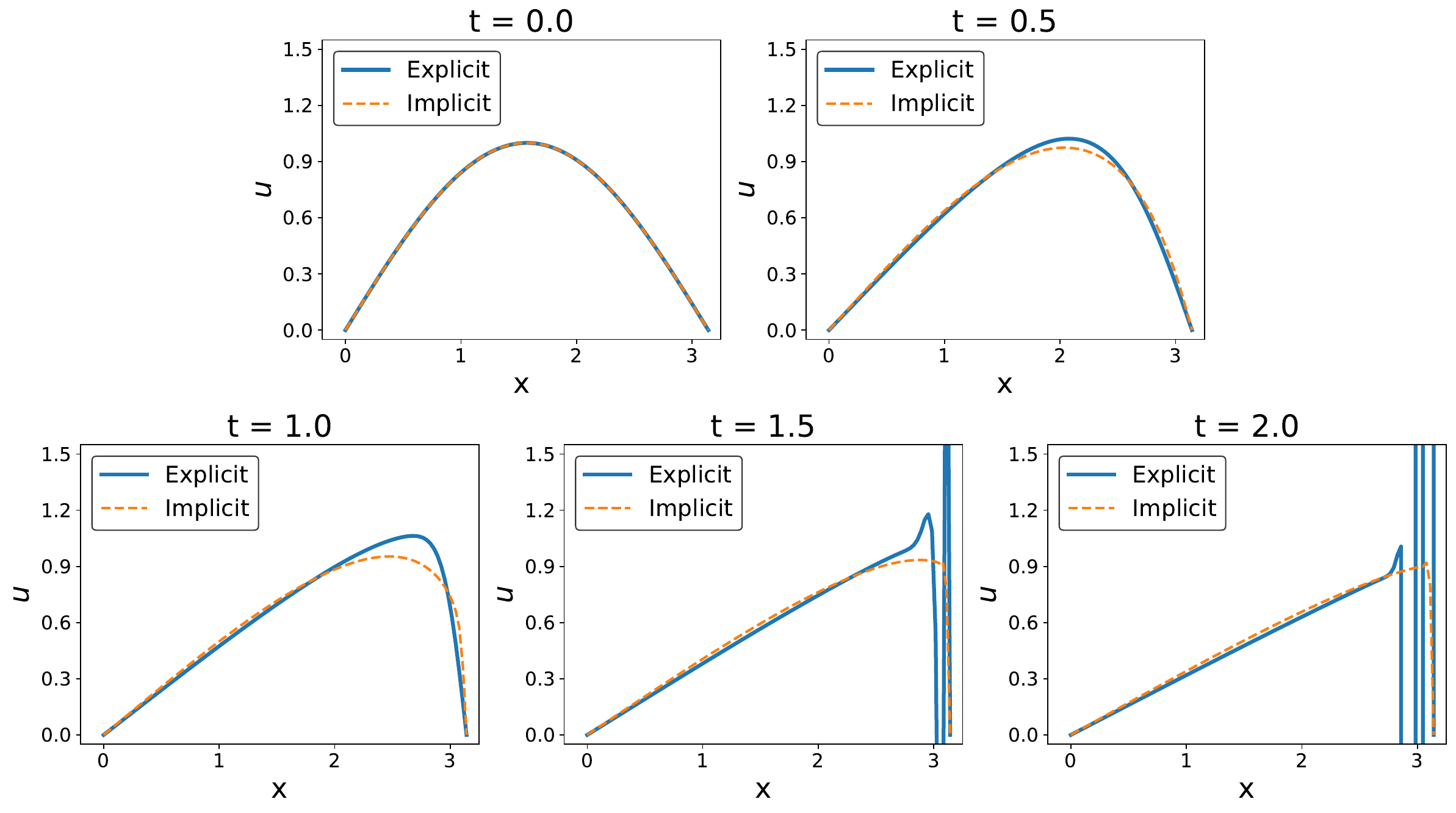}
\caption{\small The solution by explicit and implicit Lax-Fiedrichs flux of viscous Burger's equation for time step size $\delta t = 0.1$ } 
\label{fig:dynamic_1}
\end{figure}

\begin{figure}[htbp]
\centering
\includegraphics[scale=0.27]{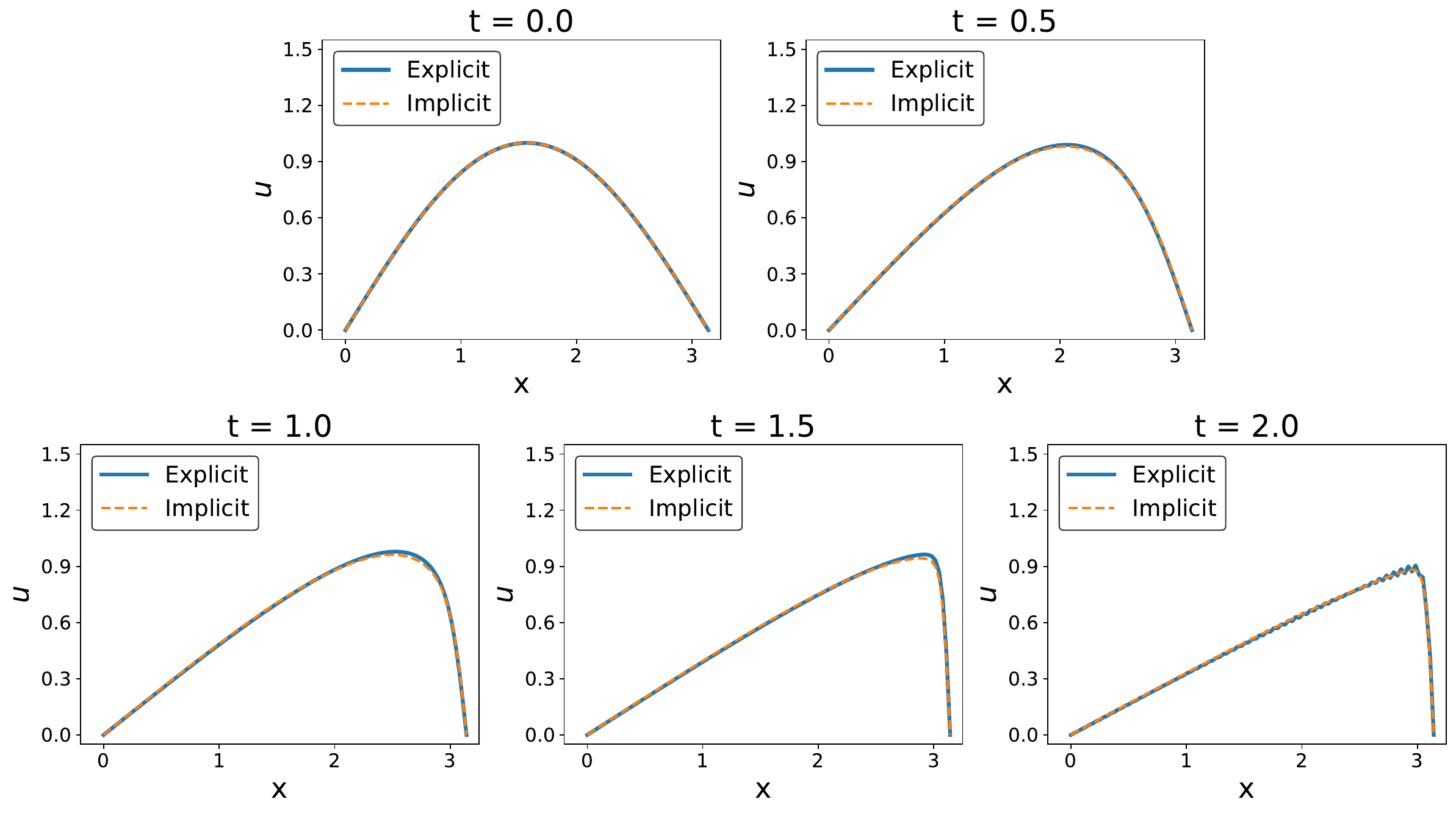}
\caption{\small The solution by explicit and implicit Lax-Fiedrichs flux of viscous Burger's equation for time step size $\delta t = 0.02$ } 
\label{fig:dynamic_2}
\end{figure}

Figure~\ref{fig:dynamic_1} and Figure~\ref{fig:dynamic_2} plot the solution of the viscous Burgers' equation for $\delta t = 0.1$ and $\delta t = 0.02$, respectively. It is clear that for $\delta t = 0.1$ the explicit Lax-Friedrichs scheme is unstable with low accuracy, while for both cases, the implicit Lax-Friedrichs scheme is stable and has high accuracy. For $\delta t = 0.02$, both methods have stable performance and the solutions are close to each other. We also plot the convergence result for the proximal ADMM method for both experiments.
\begin{figure}[htbp]
\centering
\begin{subfigure}[t]{0.45\linewidth}
\centering
\includegraphics[scale=0.19]{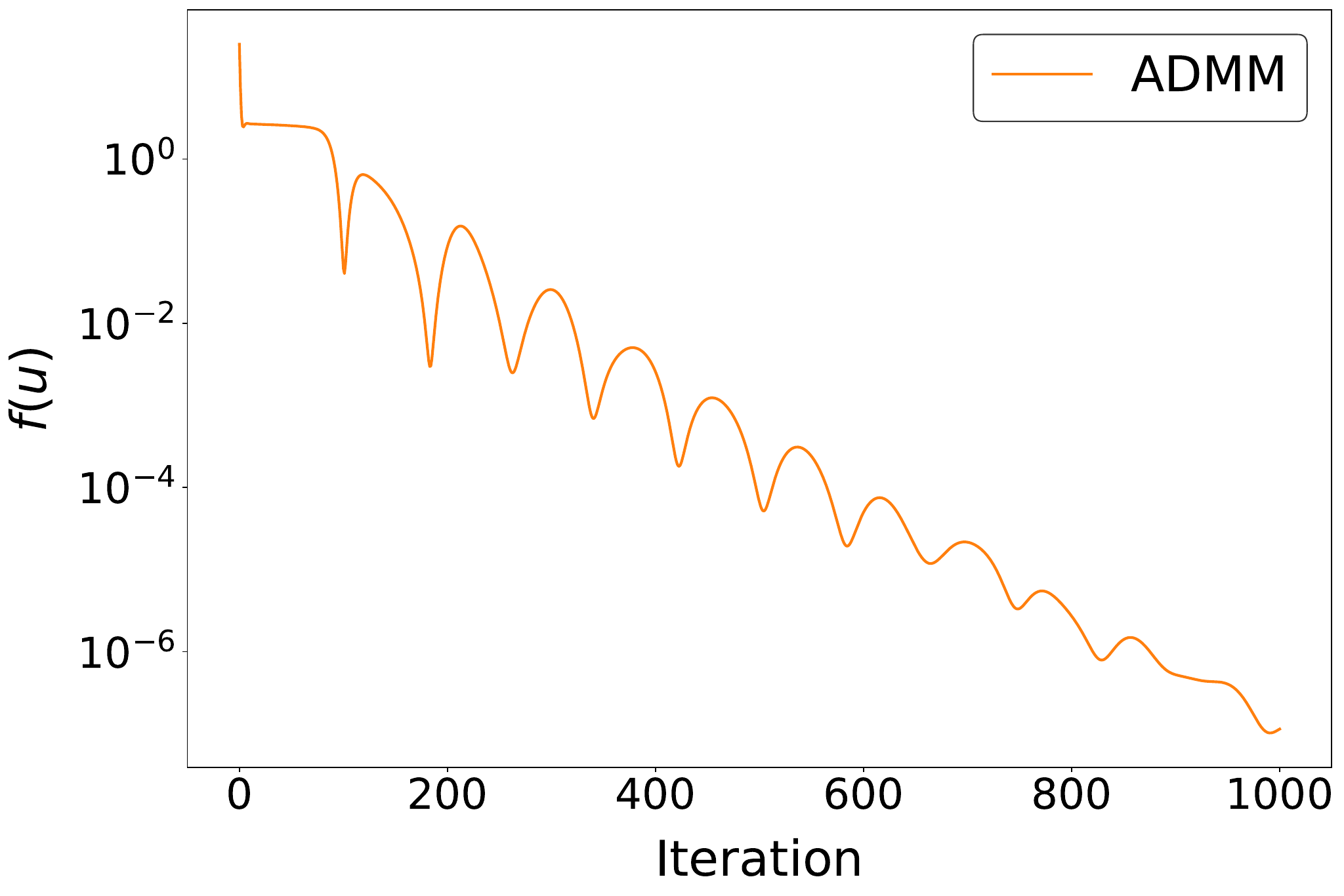}
\caption{Function Value}
\end{subfigure}
\begin{subfigure}[t]{0.45\linewidth}
\centering
\includegraphics[scale=0.19]{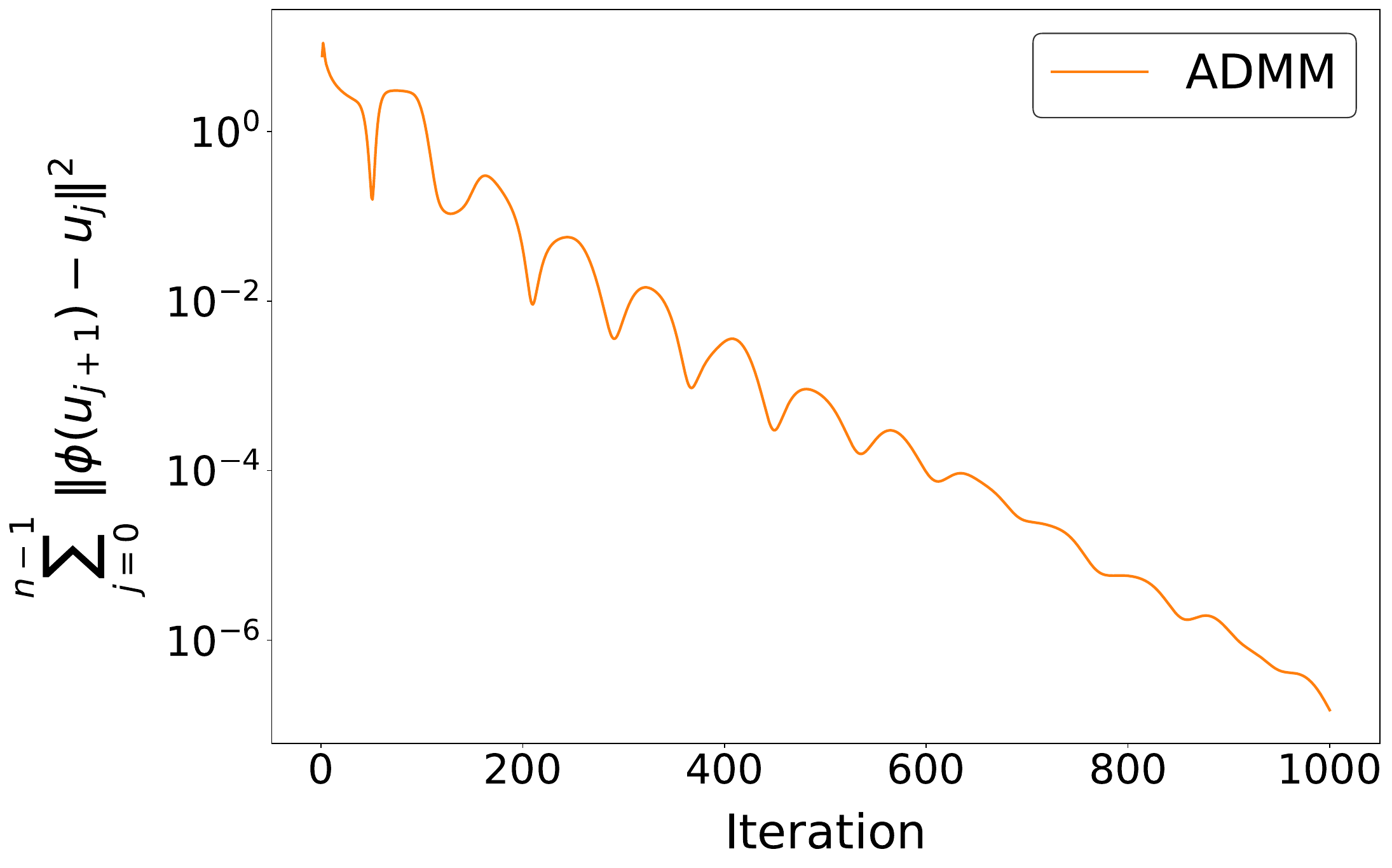}
\caption{Constraint error}
\end{subfigure}
\caption{\small The convergence graph of the proximal ADMM~\eqref{eq:imp_admm} applied to problem~\eqref{eq:imp_solver} as a solver for implicit Lax-Friedrichs scheme of the viscous Burgers' equation when $\delta t = 0.1$. } 
\label{fig:convergence_1}
\end{figure}

\begin{figure}[htbp]
\centering
\begin{subfigure}[t]{0.45\linewidth}
\centering
\includegraphics[scale=0.19]{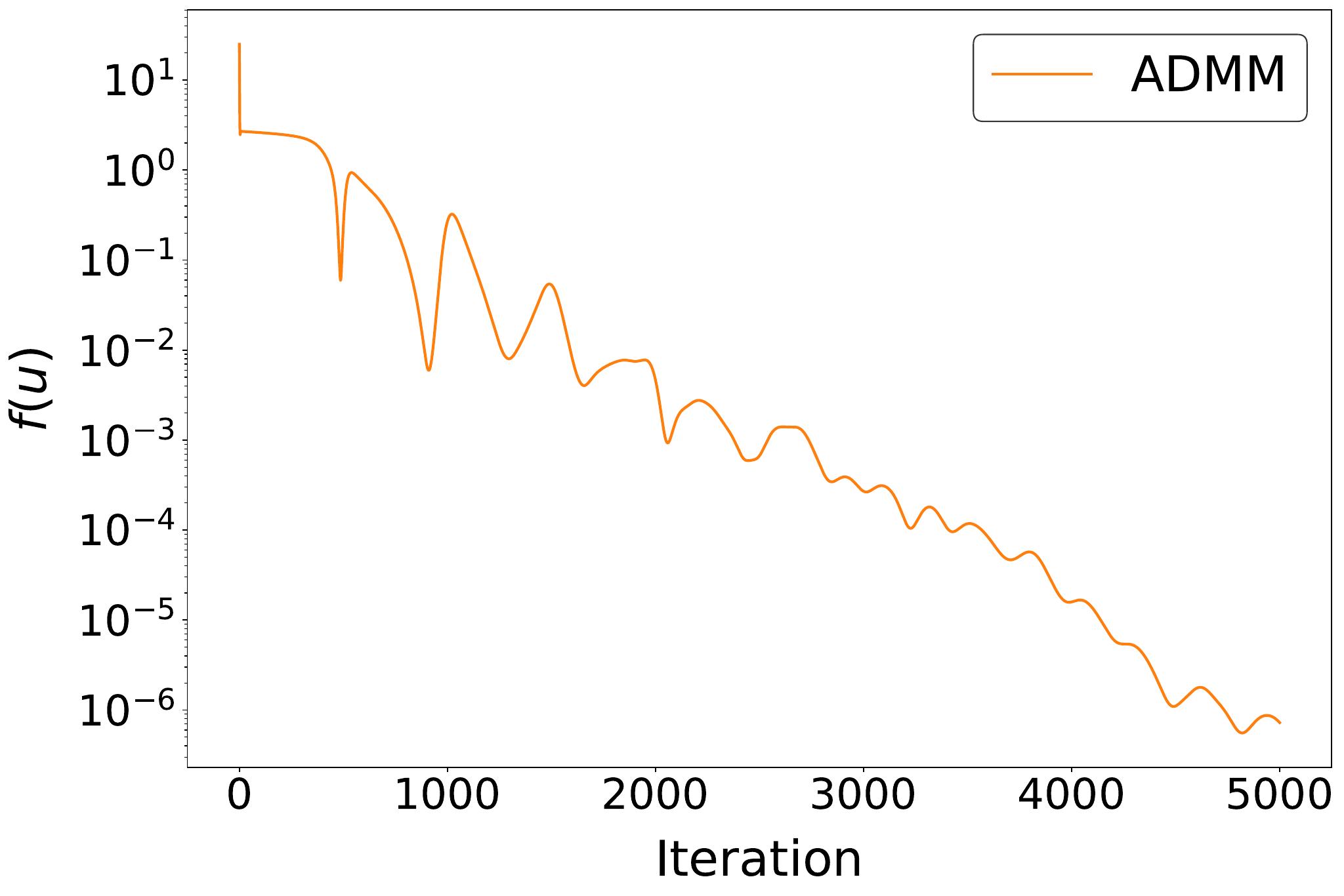}
\caption{Function Value}
\end{subfigure}
\begin{subfigure}[t]{0.45\linewidth}
\centering
\includegraphics[scale=0.19]{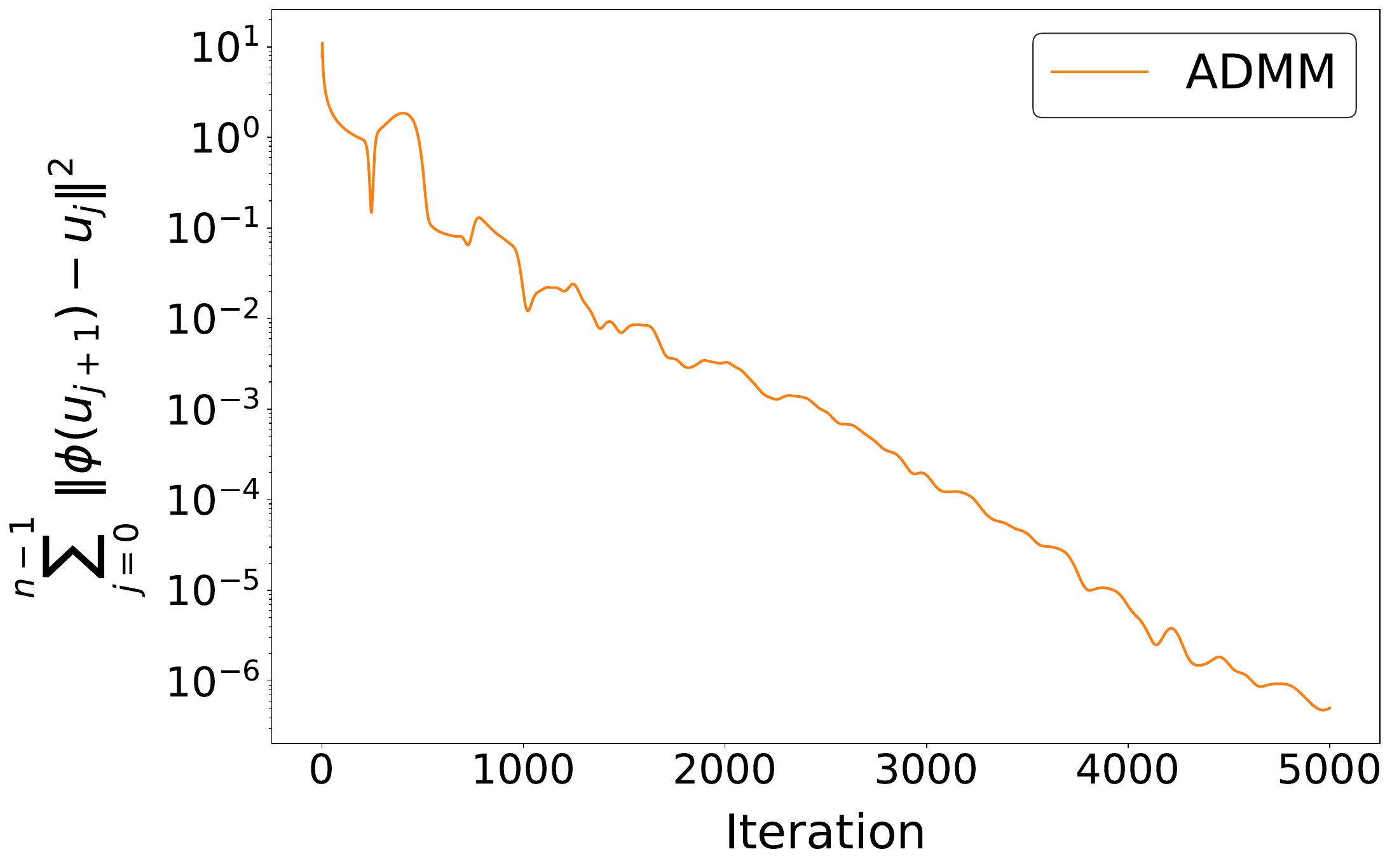}
\caption{Constraint error}
\end{subfigure}
\caption{\small The convergence graph of the proximal ADMM~\eqref{eq:imp_admm} applied to problem~\eqref{eq:imp_solver} as a solver for implicit Lax-Friedrichs scheme of the viscous Burgers' equation when $\delta t = 0.02$. } 
\label{fig:convergence_2}
\end{figure}

Figure~\ref{fig:convergence_1} and Figure~\ref{fig:convergence_2} show the convergence of the proximal ADMM. In both cases, the algorithm converges to the minimum of the problem~\eqref{eq:imp_solver}. Notice that when $\delta t = 0.1$ the algorithm converges faster than in the case when $\delta t = 0.02$. This is because for $\delta t = 0.1$, the number of time steps is only $n=20$, but for $\delta t = 0.02$, $n=100$. It is a property of the proximal ADMM algorithm, usually the smaller the number of blocks $n$, the faster it converges. It mainly influences the scalars $c_i$ and $\tilde{c}_i$~\eqref{eq:const_lya} in our theorem. In fact, not only the convergence is faster, but also the running time is shorter, because the number of subproblems to be solved decreases for each iteration.


\section{Conclusion} \label{sec:conclusion}

This paper studied a family of nonconvex optimization problems with nonlinear constraints that arise from the discretization of variational problems with dynamics constraints. We applied the proximal ADMM algorithm~\eqref{eq:admm} and analyzed the convergence. We proved that for local Lipschitz and local $L$-smooth functions, with mild additional assumptions, the sequence generated by the proximal ADMM is bounded. In addition, every accumulation point is a KKT point of the optimization problem. Furthermore, we show that there is a convergent subsequence that converges at the rate of $o(1/k)$. 

We discussed the generalization of our results to different versions of problems, including handling non-autonomous dynamical systems, implicit schemes, and optimal control problems. We also performed numerical experiment to implement the proximal ADMM algorithm to the 4D variational data assimilation problem and to solve the implicit scheme for Burgers' equation. The proximal ADMM has the advantages of choosing the initial point and finding the global minimum compared to the traditional adjoint method. It also provides a different way to solve the implicit schemes, which is formulated as an optimization problem rather than a nonlinear equation.

For future research, it is valuable to analyze the convergence properties of the ADMM algorithm with Jacobi decomposition which enables parallel implementation. The convergence guarantee is still lacking for nonconvex optimization problems with nonlinear constraints. Another direction is to find a way to dynamically adjust the parameters $\rho_i$ and $\eta_i$ to accelerate the speed of convergence. For now, the convergence is still slow. It is important to have a faster implementation in real applications.

\bibliographystyle{plain}
\bibliography{references}

\begin{thebibliography}{10}

\bibitem{bertsekas1982constrained}
Dimitri~P Bertsekas.
\newblock Constrained optimization and lagrange multiplier methods.
\newblock {\em Computer Science and Applied Mathematics}, 1982.

\bibitem{bertsekas1989parallel}
Dimitri~P Bertsekas and John~N Tsitsiklis.
\newblock Parallel and distributed computation: numerical methods, 1989.

\bibitem{betts1998survey}
John~T Betts.
\newblock Survey of numerical methods for trajectory optimization.
\newblock {\em Journal of Guidance, Control, and Dynamics}, 21(2):193--207,
  1998.

\bibitem{betts2010practical}
John~T Betts.
\newblock {\em Practical methods for optimal control and estimation using
  nonlinear programming}.
\newblock SIAM, 2010.

\bibitem{boyd2011distributed}
Stephen Boyd, Neal Parikh, Eric Chu, Borja Peleato, Jonathan Eckstein, et~al.
\newblock Distributed optimization and statistical learning via the alternating
  direction method of multipliers.
\newblock {\em Foundations and Trends{\textregistered} in Machine learning},
  3(1):1--122, 2011.

\bibitem{bryson1975applied}
Arthur~E. Bryson and Yu-Chi Ho.
\newblock {\em Applied {Optimal} {Control}: {Optimization}, {Estimation} and
  {Control}}.
\newblock Taylor \& Francis, London, 1975.

\bibitem{chen2016direct}
Caihua Chen, Bingsheng He, Yinyu Ye, and Xiaoming Yuan.
\newblock The direct extension of admm for multi-block convex minimization
  problems is not necessarily convergent.
\newblock {\em Mathematical Programming}, 155(1):57--79, 2016.

\bibitem{cohen2022dynamic}
Eyal Cohen, Nadav Hallak, and Marc Teboulle.
\newblock A dynamic alternating direction of multipliers for nonconvex
  minimization with nonlinear functional equality constraints.
\newblock {\em Journal of Optimization Theory and Applications}, pages 1--30,
  2022.

\bibitem{courtier1987variational}
Philippe Courtier and Olivier Talagrand.
\newblock Variational assimilation of meteorological observations with the
  adjoint vorticity equation. ii: Numerical results.
\newblock {\em Quarterly Journal of the Royal Meteorological Society},
  113(478):1329--1347, 1987.

\bibitem{davis2017convergence}
Damek Davis and Wotao Yin.
\newblock Convergence rate analysis of several splitting schemes.
\newblock In {\em Splitting Methods in Communication, Imaging, Science, and
  Engineering}, pages 115--163. Springer, 2017.

\bibitem{deng2017parallel}
Wei Deng, Ming-Jun Lai, Zhimin Peng, and Wotao Yin.
\newblock Parallel multi-block admm with o (1/k) convergence.
\newblock {\em Journal of Scientific Computing}, 71:712--736, 2017.

\bibitem{fan2005quadratic}
Jin-yan Fan and Ya-xiang Yuan.
\newblock On the quadratic convergence of the levenberg-marquardt method
  without nonsingularity assumption.
\newblock {\em Computing}, 74:23--39, 2005.

\bibitem{glowinski1975approximation}
Roland Glowinski and Americo Marroco.
\newblock Sur l'approximation, par {\'e}l{\'e}ments finis d'ordre un, et la
  r{\'e}solution, par p{\'e}nalisation-dualit{\'e} d'une classe de
  probl{\`e}mes de dirichlet non lin{\'e}aires.
\newblock {\em Revue fran{\c{c}}aise d'automatique, informatique, recherche
  op{\'e}rationnelle. Analyse num{\'e}rique}, 9(R2):41--76, 1975.

\bibitem{gunzburger2002perspectives}
Max~D Gunzburger.
\newblock {\em Perspectives in flow control and optimization}.
\newblock SIAM, 2002.

\bibitem{he2015full}
Bingsheng He, Liusheng Hou, and Xiaoming Yuan.
\newblock On full jacobian decomposition of the augmented lagrangian method for
  separable convex programming.
\newblock {\em SIAM Journal on Optimization}, 25(4):2274--2312, 2015.

\bibitem{he2012alternating}
Bingsheng He, Min Tao, and Xiaoming Yuan.
\newblock Alternating direction method with gaussian back substitution for
  separable convex programming.
\newblock {\em SIAM Journal on Optimization}, 22(2):313--340, 2012.

\bibitem{he20121}
Bingsheng He and Xiaoming Yuan.
\newblock On the o(1/n) convergence rate of the douglas--rachford alternating
  direction method.
\newblock {\em SIAM Journal on Numerical Analysis}, 50(2):700--709, 2012.

\bibitem{hestenes1969multiplier}
Magnus~R Hestenes.
\newblock Multiplier and gradient methods.
\newblock {\em Journal of Optimization Theory and Applications}, 4(5):303--320,
  1969.

\bibitem{hong2016convergence}
Mingyi Hong, Zhi-Quan Luo, and Meisam Razaviyayn.
\newblock Convergence analysis of alternating direction method of multipliers
  for a family of nonconvex problems.
\newblock {\em SIAM Journal on Optimization}, 26(1):337--364, 2016.

\bibitem{kelly2017introduction}
Matthew Kelly.
\newblock An introduction to trajectory optimization: How to do your own direct
  collocation.
\newblock {\em SIAM review}, 59(4):849--904, 2017.

\bibitem{le1986variational}
Fran{\c{c}}ois-Xavier Le~Dimet and Olivier Talagrand.
\newblock Variational algorithms for analysis and assimilation of
  meteorological observations: theoretical aspects.
\newblock {\em Tellus A: Dynamic Meteorology and Oceanography}, 38(2):97--110,
  1986.

\bibitem{li2025numerical}
Bowen Li and Bin Shi.
\newblock Numerical solution for nonlinear 4d variational data assimilation
  (4d-var) via admm.
\newblock {\em Journal of Computational Physics}, page 114163, 2025.

\bibitem{liavas2015parallel}
Athanasios~P Liavas and Nicholas~D Sidiropoulos.
\newblock Parallel algorithms for constrained tensor factorization via
  alternating direction method of multipliers.
\newblock {\em IEEE Transactions on Signal Processing}, 63(20):5450--5463,
  2015.

\bibitem{lorenz1963deterministic}
Edward Lorenz.
\newblock Deterministic nonperiodic flow.
\newblock {\em Journal of Atmospheric Sciences}, 20(2), 1963.

\bibitem{monteiro2013iteration}
Renato~DC Monteiro and Benar~F Svaiter.
\newblock Iteration-complexity of block-decomposition algorithms and the
  alternating direction method of multipliers.
\newblock {\em SIAM Journal on Optimization}, 23(1):475--507, 2013.

\bibitem{nocedal2006numerical}
J~Nocedal and SJ~Wright.
\newblock Numerical optimization.
\newblock {\em Springer Series in Operations Research and Financial
  Engineering}, 2006.

\bibitem{powell1969method}
Michael~JD. Powell.
\newblock A method for nonlinear constraints in minimization problems.
\newblock {\em Optimization}, pages 283--298, 1969.

\bibitem{rockafellar1976augmented}
R~Tyrrell Rockafellar.
\newblock Augmented lagrangians and applications of the proximal point
  algorithm in convex programming.
\newblock {\em Mathematics of Operations Research}, 1(2):97--116, 1976.

\bibitem{schizas2007consensus}
Ioannis~D Schizas, Alejandro Ribeiro, and Georgios~B Giannakis.
\newblock Consensus in ad hoc wsns with noisy links—part i: Distributed
  estimation of deterministic signals.
\newblock {\em IEEE Transactions on Signal Processing}, 56(1):350--364, 2007.

\bibitem{sun2024dual}
Kaizhao Sun and Xu~Andy Sun.
\newblock Dual descent augmented lagrangian method and alternating direction
  method of multipliers.
\newblock {\em SIAM Journal on Optimization}, 34(2):1679--1707, 2024.

\bibitem{sun2006optimization}
Wenyu Sun and Ya-Xiang Yuan.
\newblock {\em Optimization theory and methods: nonlinear programming},
  volume~1.
\newblock Springer Science \& Business Media, 2006.

\bibitem{talagrand1987variational}
Olivier Talagrand and Philippe Courtier.
\newblock Variational assimilation of meteorological observations with the
  adjoint vorticity equation. i: Theory.
\newblock {\em Quarterly Journal of the Royal Meteorological Society},
  113(478):1311--1328, 1987.

\bibitem{tenny2004nonlinear}
Matthew~J Tenny, Stephen~J Wright, and James~B Rawlings.
\newblock Nonlinear model predictive control via feasibility-perturbed
  sequential quadratic programming.
\newblock {\em Computational Optimization and Applications}, 28:87--121, 2004.

\bibitem{themelis2020douglas}
Andreas Themelis and Panagiotis Patrinos.
\newblock Douglas--rachford splitting and admm for nonconvex optimization:
  Tight convergence results.
\newblock {\em SIAM Journal on Optimization}, 30(1):149--181, 2020.

\bibitem{wang2019global}
Yu~Wang, Wotao Yin, and Jinshan Zeng.
\newblock Global convergence of admm in nonconvex nonsmooth optimization.
\newblock {\em Journal of Scientific Computing}, 78:29--63, 2019.

\bibitem{wen2012alternating}
Zaiwen Wen, Chao Yang, Xin Liu, and Stefano Marchesini.
\newblock Alternating direction methods for classical and ptychographic phase
  retrieval.
\newblock {\em Inverse Problems}, 28(11):115010, 2012.

\bibitem{xie2021complexity}
Yue Xie and Stephen~J Wright.
\newblock Complexity of proximal augmented lagrangian for nonconvex
  optimization with nonlinear equality constraints.
\newblock {\em Journal of Scientific Computing}, 86:1--30, 2021.

\bibitem{xu2024derivative}
Zi~Xu, Ziqi Wang, Jingjing Shen, and Yuhong Dai.
\newblock Derivative-free alternating projection algorithms for general
  nonconvex-concave minimax problems.
\newblock {\em SIAM Journal on Optimization}, 34(2):1879--1908, 2024.

\bibitem{yang2022proximal}
Yu~Yang, Qing-Shan Jia, Zhanbo Xu, Xiaohong Guan, and Costas~J Spanos.
\newblock Proximal admm for nonconvex and nonsmooth optimization.
\newblock {\em Automatica}, 146:110551, 2022.

\bibitem{yin2008bregman}
Wotao Yin, Stanley Osher, Donald Goldfarb, and Jerome Darbon.
\newblock Bregman iterative algorithms for $\ell_1$-minimization with
  applications to compressed sensing.
\newblock {\em SIAM Journal on Imaging sciences}, 1(1):143--168, 2008.

\bibitem{zhang2014asynchronous}
Ruiliang Zhang and James Kwok.
\newblock Asynchronous distributed admm for consensus optimization.
\newblock In {\em International Conference on Machine Learning}, pages
  1701--1709. PMLR, 2014.

\bibitem{zhang2011unified}
Xiaoqun Zhang, Martin Burger, and Stanley Osher.
\newblock A unified primal-dual algorithm framework based on bregman iteration.
\newblock {\em Journal of Scientific Computing}, 46:20--46, 2011.

\end{thebibliography}

\end{document}